\documentclass[10pt]{article}

\usepackage{amssymb}
\usepackage{amscd}
\usepackage{amsmath}
\usepackage{amsfonts}
\usepackage{amsthm}
\usepackage{color}
\usepackage{enumerate}
\usepackage{enumitem}

\usepackage{hyperref}

\theoremstyle{plain}

\newtheorem{theo}{Theorem}[section] 
\newtheorem{prop}[theo]{Proposition}
\newtheorem{lemme}[theo]{Lemma}
\newtheorem{cor}[theo]{Corollary}

\theoremstyle{definition}

\newtheorem{rem}[theo]{Remark}

\newcommand{\Ci}{{\mathcal{C}}^{\infty}}

\newcommand{\R}{{\mathbb{R}}}
\newcommand{\N}{{\mathbb{N}}}
\newcommand{\C}{{\mathbb{C}}}
\newcommand{\Z}{{\mathbb{Z}}}

\newcommand{\be}{{\beta}}
\newcommand{\al}{{\alpha}}
\newcommand{\la}{{\lambda}}

\newcommand{\si}{{\sigma}}
\newcommand{\hb}{{\hbar}}

\newcommand{\ga}{{\gamma}}

\newcommand{\om}{{\omega}}
\newcommand{\Om}{\Omega}

\newcommand{\Ga}{{\Gamma}}

\newcommand{\Si}{{\Sigma}}
\newcommand{\ep}{\epsilon} 

\newcommand{\op}{\operatorname}

\newcommand{\He}{\mathcal{F}} 
\newcommand{\Symb}{\mathfrak{S}} 

\newcommand{\bigo}{\mathcal{O}}    
\newcommand{\bigoinf}{\mathcal{O}_{\infty}}  

\newcommand{\con}{\overline}
\newcommand{\Wi}{\mathcal{W}}
\newcommand{\Fa}{\mathcal{A}}
\newcommand{\To}{\mathcal{T}}

\newcommand{\Di}{\mathcal{C}^{-\infty}}

\newcommand{\re}{\op{Re}}
\newcommand{\im}{\op{Im}}

\newcommand{\supp}{\op{supp}}

\newcommand{\V}{\mathcal{U}'}
\newcommand{\U}{\mathcal{U}}
\newcommand{\W}{\mathcal{U}''}

\newcommand{\Vi}{\mathcal{V}}

\newcommand{\pol}{\mathcal{P}}

\author{Laurent Charles} 
\title{Quantization of compact symplectic manifolds}

\begin{document}

\maketitle
\begin{abstract} We develop the theory of Berezin-Toeplitz operator on any compact symplectic prequantizable manifold from scratch. Our main inspiration is the Boutet de Monvel-Guillemin theory, that we simplify in several ways to obtain a concise exposition. A comparison with the spin-c Dirac quantization is also included.  
\end{abstract}

For compact K{\"a}hler manifolds, there exists a well established quantization scheme. The quantum space consists in the holomorphic sections of a prequantum bundle and to any classical observable is associated a Berezin-Toeplitz operator. This has been generalized to symplectic compact manifolds by Boutet de Monvel and Guillemin \cite{BoGu}, \cite{Gu3}, \cite{Bo3},  the basic idea being to replace the Szeg{\"o} kernel by a kernel which has similar properties. This construction has been used in some papers, see \cite{ShZe}, \cite{BoUr2} for instance.  Nevertheless, it remains difficult for two reasons.  First it is indirect: we transform a semi-classical problem, the quantization of a compact symplectic manifold equipped with a prequantum bundle $L$, into an homogeneous problem, the quantization of a symplectic cone with base the unitary bundle of $L$. Second this approach uses some sophisticated tools of microlocal analysis: Fourier integral operators with complex phase or Hermite operators. 

Our goal in this paper is to develop the quantization of symplectic compact manifolds from scratch, in a direct way and without using any substantial result of microlocal analysis.
So we introduce a class of spaces similar to the spaces of holomorphic sections, prove their existence and establish the basic results for the corresponding Berezin-Toeplitz operators. 

Another approach for the quantization of symplectic manifold is through spin-c Dirac operator, cf. \cite{Du} for an introduction and \cite{MaMa2008} for the study of Berezin-Toeplitz operators in this context. We will explain how this enters in our setting.  

\section{Statement of the results}

Let $(M, \om)$ be a symplectic compact manifold with dimension $2n$. Let $j$ be an almost complex structure compatible with $\om$, that is $\om ( jX, jY) = \om ( X, Y)$ for any $X, Y \in TM$ and $\om ( X, j X) >0$ if $X \neq 0$. Assume that the cohomology class $\frac{1}{2\pi} [\om]$ is integral and choose a Hermitian line bundle $L \rightarrow M$ with a connection $\nabla$ of curvature $\frac{1}{i} \om$. Such a pair $(L, \nabla)$ is called a prequantum bundle.  Let $A \rightarrow M$ be a Hermitian vector bundle. 

\subsection*{Existence of the projector} 
For any integer $k$, set $A_k = L^k \otimes A$ and introduce the scalar product on  $\Ci (M ,A_k)$ given by integrating the pointwise scalar product against the Liouville measure. Let $\mathcal{H}_k$  be a finite dimensional subspace of $\Ci (M ,A_k)$. Consider the orthogonal projector of $\Ci ( M , A_k)$ onto $\mathcal{H}_k$ and its Schwartz kernel $\Pi_k  \in \Ci ( M^2, A_k \boxtimes \con{A}_k)$. This kernel may be defined in a elementary way by the formula
$$ \Pi_k ( x, y) = \sum_{i =1 }^{N_k } f_i (x) \otimes \con{f_i ( y)} , \qquad  \forall (x,y) \in M^2$$
where $(f_i , i =1, \ldots , N_k)$ is any orthonormal basis of $\mathcal{H}_k$.

\begin{theo} \label{theo:existence_proj}
For any symplectic compact manifold $(M, \om)$ with a compatible almost complex structure $j$, a prequantum bundle $L \rightarrow M$ and a Hermitian vector bundle $A \rightarrow M$, there exists a family $( \mathcal{H}_k \subset \Ci ( M , L^k \otimes A) , \; k \in \N)$ of finite dimensional subspaces such that the corresponding family $(\Pi_k)$ of Schwartz kernels is in $\bigoinf ( k^n )$ and satisfies for any $m \in \N$,  
\begin{gather}\label{eq:kerprojdevas}
 \Pi_k ( x,  y ) = \Bigl( \frac{k}{2 \pi} \Bigr) ^n E^k ( x, y) \sum_{\ell \in \Z \cap [ -m , m/2]} k^{-\ell} \si_{\ell} ( x, y) + \bigo ( k^{-(m+1)/2}) 
\end{gather} 
where $2n$ is the dimension of $M$ and 
\begin{itemize} 
\item $E$ is a section of $L \boxtimes \con{L}$ satisfying $E(x,x) = 1$, $|E (x, y )| < 1$ if $ x \neq y$, $E(y,x) = \con{E}(x,y)$ for any $x,y$ and for any vector field $Z \in \Ci (M ,T^{1,0} M)$, $(\nabla_{\con{Z}} \boxtimes \op{id} )E$  vanishes to second order along the diagonal of $M^2$. 
\item For any $\ell \in \Z$, $\si_\ell$ is a section of $A \boxtimes \con{A}$. If $\ell$ is negative, $\si_\ell$ vanishes to order $-3\ell$ along the diagonal.   
\end{itemize} 
Furthermore $\si_{0} (x, x)  = \op{id}_{A_x}$ for any $x \in M$. 
\end{theo} 

In this statement, we have made the identification $L_x \otimes \con{L}_x \simeq \C$ and $A_x \otimes \con{A}_x \simeq \op{End} A_x$ induced by the Hermitian metrics. The meaning of $\bigo ( k^{-N}) $ and $\bigoinf ( k^{n})$ is explained in Sections \ref{sec:asympt-expans} and \ref{sec:derivatives-control}. Let us just say now that the $\bigo (k^{-N}) $ are uniform on $M^2$ and that the set $\bigoinf ( k^{n})$ consists of families in $\bigo ( k^{n})$ whose successive derivatives are controlled in a precise way. 

The dimension of $\mathcal{H}_k$ is given by integrating the function $x \rightarrow \op{tr} (\Pi_k (x,x))$ against the Liouville measure. This leads to the following estimate:
$$ \op{dim} \mathcal{H}_k = \Bigl( \frac{k}{2 \pi} \Bigr)^n  ( \op{rank} A) \op{vol}( M) + \bigo ( k^{n-1}),$$
where the volume of $M$ is by definition the integral of $\om^n/ n!$.

\begin{rem} \label{rem:kahler}
Assume that $j$ is integrable, so that $M$ is a K{\"a}hler manifold. Then $L$ has a unique holomorphic structure such that $\nabla^{0,1} = \con {\partial}$. Assume that $A$ has a holomorphic structure. Then we can consider the space $\mathcal{H}_k  = H^0 (M, A_k)$ of holomorphic section of $A_k$. In the case where $A$ is the trivial line bundle, it is deduced in \cite{oim_op} from the seminal paper \cite{BoSj} that the sequence $( H^0 ( M , A_k) , \; k \in \N)$ satisfies the conditions of Theorem \ref{theo:existence_proj}. A direct proof of this result has been given in \cite{BeBeSj} which includes the case of any Hermitian holomorphic bundle $A$. Similar results for $A=\C$ in both the holomorphic and symplectic cases are proved in \cite{ShZe}.

Actually, in the K\"ahler case, one has more precise estimates. First, we may choose the section $E$ so that $(\nabla_{\con{Z}} \boxtimes \op{id} )E$ and $( \op{id} \boxtimes \nabla_Z ) E$ vanish to infinite order along the diagonal of $M^2$. With this choice, we can assume that for any $\ell <0$, the section $\si_\ell$ is null. Furthermore, for any $\ell \geqslant 0$ and for any $Z \in \Ci ( M , T^{1,0}M) $,  $(\nabla_{\con{Z}} \boxtimes \op{id} )\si_\ell $ and $( \op{id} \boxtimes \nabla_Z ) \si_{\ell}$ vanish to infinite order along the diagonal of $M^2$. \qed
\end{rem}

The construction of $(\Pi_k)$ is as follows: consider any sections $E$ and $\si_0$ satisfying the conditions given in Theorem \ref{theo:existence_proj} and let $P_k$ be the operator with Schwartz kernel $\bigl( \frac{k}{2\pi} \bigr)^n E^k \si_0 $.  Assume furthermore that $\si_0 ( y,x) = \con{\si}_0 (x,y)$  so that $P_k$ is self-adjoint. One proves that the spectrum $P_k$ concentrates onto 0 and 1, more precisely $ \op{spec} (P_k) \subset [-Ck^{-1/2} , C k^{-1/2}] \cup [ 1 - C k^{-1/2}, 1+ C k^{-1/2}]$ where $C$ is a positive constant. Then we define $\Pi_k$ for large $k$ by $\Pi_k = f ( P_k)$, with $f \in \mathcal{C} ( \R, \R)$ any function equal to $0$ (resp. 1) on a neighborhood of $0$ (resp. 1).

\subsection*{Toeplitz operators}

Consider a family $\mathcal{H}= ( \mathcal{H}_k \subset \Ci ( M , A_k) , \; k \in \N)$ of subspaces satisfying the conditions of Theorem \ref{theo:existence_proj}. A {\em Toeplitz operator}  is any family $(T_k : \mathcal{H}_k \rightarrow \mathcal{H}_k, \; k \in \N)$ of operators of the form 
\begin{gather} \label{eq:def_Toep}
 T_k = \Pi_k f( \cdot, k ) + R_k , \qquad k \in \N^*
\end{gather}
where $f(\cdot, k )$, viewed as a multiplication operator, is a sequence in $\Ci ( M, \op{End} A )$ admitting an asymptotic expansion $ f_0 + k^{-1} f_1 + \ldots $ for the $\Ci $ topology. Furthermore the norm of $R_k \in \op{End} \mathcal{H}_k$ is a $\bigo ( k ^{-N})$ for any $N$. 
We denote by $\To (M , L, A, \mathcal{H})$ the space of Toeplitz operators.

\begin{theo} \label{theo:intro_contravariant}
The space $\To=\To (M , L, A, \mathcal{H}) $ is closed under the formation of product. So it is an algebra with identity $(\Pi_k)$. The symbol map $$\si_{\op{cont}} : \To \rightarrow  \Ci ( M , \op{End} A) [[\hb ]]$$ sending $(T_k)$ into the formal series $ f_0 + \hb f_1 + \ldots $ where the functions $f_i$ are the coefficients of the asymptotic expansion of the multiplicator $f( \cdot , k)$ is well defined. It is onto and its kernel is the ideal consisting of $\bigo ( k^{-\infty})$ Toeplitz operators. More precisely, for any integer $\ell$,
$ \| T_k \| = \bigo ( k^{-\ell})$ if and only if $\si_{\op{cont}} ( T) = \bigo ( \hb^{\ell}) .$ 
\end{theo} 

According to Berezin terminology, we call $\si_{\op{cont}} (T) $ the {\em contravariant} symbol of $T$. 
We can also define in this context a {\em covariant} symbol, cf. Section \ref{sec:toeplitz-operators}. 
The {\em principal} symbol $\si_0 (T) \in \Ci ( M , \op{End} A)$ is by definition the first coefficient of the contravariant symbol, so $\si_{\op{cont}} (T) = \si_0 (T) + \bigo ( \hb)$. 
The properties of the principal symbol are summarized in the following theorem.

\begin{theo} \label{theo:intro:principal}
For any Toeplitz operators $T, S \in  \To ( M, L , A, \mathcal{H}) $ with principal symbols $f$ and $g$,  we have
$$ \si_{0} ( T S) = f.g  .$$
Consequently, if $f$ or $g$ is scalar valued, then $\si_0 ( [ T,S] ) =0$ so that $ ik [T,S] $ is a Toeplitz operator. In the case where $f$ and $g$ are scalar valued, we have
$$ \si_0   (ik [T,S] ) = \{ f, g \} .$$
Denoting by $T_k (x,x)\in A_X \otimes \con{A}_x $ the value of the Schwartz kernel of $T_k$ at $(x,x)$, we have
$$ T_k ( x, x) = \Bigl( \frac{k}{2 \pi} \Bigr)^n (f (x) + \bigo ( k^{-1}) ) $$
Denoting by $\| T_k \|$ the operator norm of $T_k$ corresponding to the scalar product of $\mathcal{H}_k \subset \Ci ( M , A_k)$, we have
$$  \| T_k \|  = \sup_{y \in M} | f (y) | + \bigo ( k^{-1}) $$ 
where for any $y \in M$, $| f (y) |$ is the operator norm of $f (y) \in \op{End} A_y $.
\end{theo} 

The computation of the principal symbol of $ik [T , S]$ in the case where only $f$ is scalar valued is more delicate, because it depends on the class of $T$ modulo $\bigo( k^{-2})$, cf. Section \ref{sec:commutators}.

\begin{rem}In the K{\"a}hler case, cf. Remark \ref{rem:kahler}, with $A = \C$, Theorem \ref{theo:intro_contravariant} and Theorem \ref{theo:intro:principal} have been deduced in \cite{BoMeSc} from the theory of \cite{BoGu}, cf. also \cite{oim_op} and \cite{MaMa2012} for different approaches, this last paper treats also the case of any holomorphic vector bundle $A$. Similar results in the general symplectic case are proved in \cite{Gu3} and \cite{MaMa2008}.
\end{rem}

In the case where $A$ has rank one so that $\op{End} A$ is the trivial bundle on $M^2$, we will prove some further results on the contravaraint symbols. Denote by $\star_{\op{cont}}$ the product of $\Ci ( M) [[\hbar]]$ giving the composition of contravariant symbol, $\si_{\op{cont} } (TS) = \si_{\op{cont} } (T) \star_{\op{cont} } \si_{\op{cont} } (S)$. Since $\si_{\op{cont}} ( k^{-1} T) = \hbar \si_{\op{cont}} ( T)$, one easily see that the product $\star_{\op{cont}}$ has the form
$$  \Bigr( \sum_{\ell \in \N} \hb^\ell f_\ell \Bigr) \star_{\op{cont}}  \Bigl( \sum_{m \in \N }  \hb^m g_m \Bigr) = \sum_{r\in \N}  \hb^r \sum_{m+ \ell + p = r } B_p ( f_\ell, g_m )  
$$ 
where each $B_{\ell}$ is bilinear operator $\Ci (M) \times \Ci ( M) \rightarrow \Ci (M)$. By Theorem \ref{theo:intro:principal}, 
$$B_0 ( f,g) = fg , \qquad B_1 ( f,g) - B_1 ( g,f) = \frac{1}{i} \{ f,g \}. $$ 
Denote by $T_{k} (f)$ the Toeplitz operator with multiplicator $f$. By Theorem \ref{theo:intro_contravariant}, for any $N \in \N$, one has for any $f,g \in \Ci (M)$
$$ T_k (f) T_k ( g)  = \sum_{\ell =0}^{N} k^{-\ell} T_k (  B_{\ell} ( f,g) ) + \bigo ( k^{-N-1})$$ 
the $\bigo$ being in uniform norm. In our last result we make explicit the dependence of this $\bigo$ in terms of $f$ and $g$. Introduce a Riemannian metric in $M$, and denote by $|f|_\ell$ the corresponding $\mathcal{C}^\ell$ norm of a function $f \in \Ci (M)$. Define
$$|f,g|_N = \sum_{\ell =0}^{N} |f|_{\ell} |g|_{N - \ell}$$
for any $f,g \in \Ci (M)$. 

\begin{theo} \label{theo:intro:remainder}
For any $\ell$, $B_{\ell}$ is a bidifferential operator of order $2 \ell$. For any $N \in \N$, there exists $C_N$ such that for any $f, g \in \Ci ( M)$, one has
$$ \Bigl\| T_k ( f) T_k (g) - \sum_{\ell =0}^{N} k^{-\ell} T_k (  B_{\ell} ( f,g) ) \Bigr \| \leqslant C_N k^{-(N+1)}|f,g|_{2(N+1)}.$$
\end{theo}
This result is a useful tool when we work at small scale, cf. for instance \cite{oimPo2}. Another potential application is Egorov Theorem up to Ehrenfest time or any other computation for semiclassical pseudo-differential operator involving the exotic symbol classes. 

\begin{rem}
Theorem \ref{theo:intro:remainder} for $N=0$ was proved in \cite{oimPo1} together with the estimate 
$$ \bigl\| [ T_k(  f)  , T_k  (  g )  ] - \frac{i}{k} T_k (  \{ f, g \} )  \bigr\| \leqslant C k^{-2} ( |f|_1 |g|_3 + |f|_2 |g|_2 + |f|_3 |g|_1 ) .$$ 
Weaker estimates were previously obtained in \cite{BaMaMaPi}.   
\end{rem} 

\subsection*{Comparison with other approaches} 

\subsubsection*{Spin-c Dirac operators} 

First the previous results can be generalized as follows: Let $( M, j , L , A)$ be  as above and assume that $A$ is a subbundle of an Hermitian bundle $\mathbf{A}$. Set $\mathbf{A}_k = L^k \otimes \mathbf{A}$ and consider any family $( \mathcal{H}_k \subset \Ci ( M , \mathbf{A}_k ), \; k \in \N)$ of finite dimensional subspaces, satisfying exactly the same condition as in Theorem \ref{theo:existence_proj} with $\mathbf{A}$ replacing $A$, except that $\si_0 (x, x) = \pi_A (x)$ where $\pi_A(x) \in \op{End} \mathbf{A}_x$ is the orthogonal projector onto $A_x \subset \mathbf{A}_x$. We can define the corresponding Toeplitz operators, as we did in (\ref{eq:def_Toep}), with multiplicators in $\Ci ( M , \op{End} A)$. Then Theorems \ref{theo:intro_contravariant} and \ref{theo:intro:principal} hold with these data. 

This generalization allows us to compare our constructions with the spin-c Dirac quantization. In that case, we start with $(M, L, j , A)$ and set $\mathbf{A} = A \otimes S$ where $S = \wedge ( (T^*M)^{1,0})$ is the Spinor bundle. Writing $S = \C \oplus \wedge^{>0} ( (T^*M)^{1,0}) $, we can view $A$ as a subbundle of $\mathbf{A}$.   Choosing any Hermitian connection on the canonical bundle on $M$, we obtain a spin-c Dirac operator $D_k : \Ci ( M , L^k \otimes \bold{A}) \rightarrow \Ci ( M , L^k \otimes \mathbf{A})$. Then we will deduce from the results of Dai, Liu and Ma \cite{DaLiMa} that the Schwartz kernel of the projector onto $\mathcal{H}_k = \op{Ker} D_k$ satisfies the condition of Theorem \ref{theo:existence_proj} with $\si_0 ( x, x) = \pi_A (x)$, cf. Theorem  \ref{theo:spin-c-dirac}. So the spin-c Dirac quantization may be viewed as a particular case of our theory.

\subsubsection*{Boutet de Monvel-Guillemin theory and other works}

The main difference with the theory in \cite{BoGu}, is that first we have a direct semi-classical approach and second we do not use Fourier integral operators or Hermite operators. Instead of that, we consider an algebra of operators, denoted  by $\Fa ( M , L , j, A)$, which consists of families $(P_k : \Ci( M , A_k ) \rightarrow \Ci ( M , A_k), k \in \N)$  of operators whose Schwartz kernels satisfy the conditions of Theorem \ref{theo:existence_proj}, except for $\si_0(x,x)$ which can be arbitrary. We will see that this algebra has a natural filtration.  The main result, Theorems \ref{theo:main1new} and \ref{theo:main1new2},  is the computation  of the product of the symbols corresponding to this filtration. The symbol composition law is rather algebraic and has the particularity of being non commutative.
This is the use of this algebra which shorten the theory. 
The origin of this algebra can be found in the papers \cite{oim_op}, \cite{oim_demi} \cite{oim_lag} devoted to the K\"ahler case. 

We can also find similar semi-classical descriptions of the projector in \cite{ShZe}, section 3. In this paper, the special frame $E$ of Theorem \ref{theo:existence_proj} appears implicitely through the Heisenberg coordinates. This approach has also been adapted for the quantization of Lagrangian submanifolds in \cite{Pa}, cf. also \cite{DePa} and \cite{BuGuWa}.  

An important feature of the constructions in \cite{BoGu} and \cite{ShZe} is the fact that the spaces $\mathcal{H}_k$ are the cohomology groups of a complex generalising the classical Dolbeault complex. We won't address this topic here. 

\subsection*{Outline of the paper}

Sections \ref{sec:fourier-families} and \ref{sec:algebra} are devoted to the algebra  $\Fa ( M , L , j, A)$. 
In section \ref{sec:existence-projector}, we prove Theorem \ref{theo:existence_proj}. In Section \ref{sec:toeplitz-operators}, we define and prove the basic facts on Toeplitz operators, in particular Theorems \ref{theo:intro_contravariant} and \ref{theo:intro:principal}. Theorem \ref{theo:intro:remainder} is proved in Section \ref{sec:further-estimates}. 
The paper ends with an appendix on spin-c Dirac quantization.

\section{On a class of section families} \label{sec:fourier-families}

\subsection{Asymptotic expansions} \label{sec:asympt-expans}

Let $M$ be a manifold.   Introduce a Hermitian line bundle $L \rightarrow M$ and a Hermitian vector bundle $A \rightarrow M$. We call $A$ the auxiliary bundle. For any integer $k$, let $A_k = L^k \otimes A$. Consider a family 
\begin{gather} \label{eq:famA_k}
\Psi= (\Psi ( \cdot, k) \in  \Ci ( M , A_k), \; k \in \N^* ).
\end{gather} 
We say that $\Psi $ is a $\bigo ( k ^{-p})$ if for any compact set $K$ of $M$, there exists $C$ such that
\begin{gather} \label{eq:bigokm}
 | \Psi (x, k ) | \leqslant C k^{-p} , \qquad  \forall \, x \in K, \, \forall \, k \in \N^*.
\end{gather}
Here $|\Psi ( x , k)|$ denotes the pointwise norm of $\Psi ( \cdot, k)$ at $x$. We need the following version of Borel lemma.

\begin{prop}  \label{prop:Borel}
For any $m
\in \N$, let  $(\Psi_m (
\cdot, k))$ be a family in $\bigo ( k^{-m})$. Then there exists
a family $ (\Psi (\cdot , k))$ in $\bigo (1)$ such
that for any $m$, we have
$$ \Psi ( \cdot, k ) = \sum_{\ell = 0 }^{m-1} \Psi_{\ell} ( \cdot,
k)  + \bigo ( k ^{- m} ) .$$
\end{prop}

\begin{proof} 
The proof is based on the same argument showing the existence of a function with a prescribed Taylor expansion. Since this kind of proof is standard in microlocal analysis, we only give a sketch.  We introduce a function $\chi \in \Ci _0 (\R)$ such that $\chi =1 $ on $[0,1]$. We set $ \Psi ( \cdot , k ) = \sum \chi ( \la_m / k ) \Psi _m ( \cdot, k ) $
where the sequence $\la_m \rightarrow \infty$ has to be fixed. Choose an exhausting sequence $(K_j)$ of compact sets of $M$. By a diagonal argument, if we choose the $\la_m$ sufficiently large, we have for any $j$ and any $m \geqslant j$,  
$$   | \chi ( \la_m /k ) \Psi ( x, k ) | \leqslant k^{-m+1} 2^{-m} , \qquad \forall x \in K_j.$$
With this choice,  $ \Psi ( \cdot , k )$ has the required asymptotic expansion.
\end{proof}

\subsection{Derivatives control} \label{sec:derivatives-control}

Consider first a sequence $(\Psi (\cdot, k), k \in \N^*)$ of $\Ci(M)$. As above we say that $(\Psi (\cdot, k))$ is in $\bigo( k^{-m})$ if for any compact set $K$ of $M$, there exists $C$ such that (\ref{eq:bigokm}) is satisfied. 
We say that $ (\Psi (\cdot, k))$ is a $\bigoinf (k^{-m})$ 
if  for any $\ell \in \N$ and any vector fields $X_1,
\ldots, X_\ell$ of $M$, 
$$X_1 \ldots X_\ell \Psi ( \cdot , \tau ) \in \bigo( k^{\ell - m}).$$
Observe that at each derivative we loose one power of $k$. The
reason for this is that we want the definition to be invariant by multiplication by $\exp ( i k h)$ where $h \in \Ci(M,
\R)$. Indeed it is easy to check that
\begin{gather} \label{eq:mult_phase}
 (\Psi ( \cdot, k)) \in \bigoinf( k^ {-m}) \Leftrightarrow (
e^{i k h} \Psi (\cdot, k)) \in \bigoinf( k^ {-m}) 
\end{gather} 
With this property, we  will extend the definition to families of
sections of bundles. 

Assume now that $\Psi$ is a family of the form (\ref{eq:famA_k}). 
We say that $\Psi $ is  a $\bigoinf (k^{-m})$ if for any point $p \in M$, there exists an open neighborhood $U$ of $p$ and local unitary frames $\si$ and $(\tau_i, \; i=1, \ldots , r ) $ of $L$ and $A$ respectively defined on $U$, such that the families $(f_i ( \cdot, k))$, $i=1, \ldots r$ of $\Ci ( U, \C)$ given by 
$$ \Psi ( \cdot, k )  = \sum_{i=1}^r f_i( \cdot , k) \si^k \otimes \tau_i ,$$
are in $\bigoinf ( k^{-m})$. Because of the equivalence  (\ref{eq:mult_phase}), this definition does not depend on the choice of the local frames $\si$ and $(\tau_i)$.

Let us state the corresponding Borel lemma. 

\begin{prop}  \label{prop:Borel_bigohalf}
For any $m \in \N$, let  $(\Psi_m (
\cdot, k))$ be a family of the form (\ref{eq:famA_k}) in $\bigoinf (k^{-m})$. Then there exist
a family $ (\Psi (\cdot , k))$ in $\bigoinf ( 1)$ such
that for any $m$, we have
$$ \Psi ( \cdot, k ) = \sum_{\ell = 0 }^{m-1} \Psi_{\ell} ( \cdot,
k)  + \bigoinf ( k ^{-m} ) .$$
\end{prop}

The proof is similar to the one of Proposition \ref{prop:Borel}. Another useful result is the following. 

\begin{prop} \label{prop:derivatives-control}
Let $\Psi$ be a family of the form (\ref{eq:famA_k}) such that for any $m$
$$  \Psi ( \cdot, k ) = \sum_{\ell = 0 }^{m-1} \Psi_{\ell} ( \cdot,
k)  + \bigo ( k ^{-m} ) .$$
where for any $\ell \in \N$, $\Psi_\ell \in \bigoinf (k^{-\ell})$. Assume furthermore that $\Psi \in \bigoinf ( k^{N})$ for some $N$. Then we have 
$$  \Psi ( \cdot, k ) = \sum_{\ell = 0 }^{m-1} \Psi_{\ell} ( \cdot,
k)  + \bigoinf ( k ^{-m} ) $$
for any $m$. 
\end{prop} 

Again the proof is a standard argument, cf. as instance Lemma 3.2 of \cite{Sh}.  
In particular, Proposition \ref{prop:derivatives-control} says that for any $N$, we have 
$$\bigoinf ( k^N ) \cap \bigo ( k^{-\infty}) = \bigoinf ( k^{-\infty})$$ 
where we use the notations 
\begin{gather} \label{eq:not_kinfty}  \bigo ( k^{-\infty}) = \bigcap _{ \ell \in \N} \bigo ( k^{- \ell}), \qquad  \bigoinf ( k^{-\infty}) = \bigcap _{ \ell \in \N} \bigoinf ( k^{- \ell}).
\end{gather}

\subsection{Gaussian weight}

Let $S \in \Ci(M)$. Assume that $\varphi = - 2 \re S$ satisfies 
\begin{enumerate}[label=(\arabic{section}.\arabic{subsection}.\roman{*}), ref=(\arabic{section}.\arabic{subsection}.\roman{*})]
\item \label{item:positif-zero}
$\varphi \geqslant 0$ and $\Si = \{ \varphi = 0 \}$ is a submanifold of $M$,
\item \label{item:hessian}
the restriction of the Hessian of $\varphi$ to the normal bundle of $\Si$ is non degenerate.
\end{enumerate}
Note that the first condition implies that the differential of
$\varphi$ vanishes along $\Si$ and the Hessian of
$ \varphi $  vanishes in the directions tangent to $\Si$. So
$\operatorname{Hess} ( \varphi )$ factorizes to a non negative quadratic form of the normal bundle of $\Si$. The second condition says that this quadratic form is definite, that is for any $p \in \Si$ and $X \in T_p M$,
$$ \operatorname{Hess} ( \varphi ) (X) = 0 \quad \Rightarrow \quad X \in T_p \Si $$
In the sequel we will study the asymptotic behavior of families of
the form $( e^{-\tau S} f$, $\tau \geqslant 1 )$, with $f$ a smooth
function. The asymptotic properties of such a family only depend on the Taylor expansion of the
amplitude $f$ along $\Si$.  

For any positive integer $N$, we say a function $f \in \Ci(M)$ vanishes to order $N$ along $\Si$, if for any integer $m$ such that $0 \leqslant m < N$, for any vector fields $X_1$, \ldots,  $X_m$ of $M$ 
$$ X_1 \ldots X_m f = 0 \qquad \text{on } \Si . $$
We use the notation 
$$ f =g + \bigo (N) \qquad \text{ along } \Si$$ 
to say that $f- g $ vanishes to order $N$ along $\Si$. The basic property that we will need is that a function $f\in \Ci _0 (M)$ vanishes to order $N$ along $\Si$ if and only if there exists $C>0$ such that for any $x \in M$, we have $| f ( x) | \leqslant C (\varphi(x))^{N/2}$. 

\begin{prop} \label{prop:estim_symb} 
Let $S \in \Ci (M)$ which satisfies  conditions
\ref{item:positif-zero} and \ref{item:hessian}. Let $\ell \in \N$ and $f \in \Ci_0 (M)$ which vanishes to order $\ell$
along $\Si$ . Then there exists $C$ such that 
\begin{gather} \label{eq:1}
  \bigl| e^{ - \tau S(x) } f(x)  \bigr| \leqslant C \tau^{-\ell/2} ,
\qquad \forall \, \tau \geqslant 1 , \, \forall \, x \in M.
\end{gather}
Let $N \in \N$ and $f_0, \ldots , f_N \in \Ci_0 (M)$ such that for any
$p= 0, \ldots . N$, $f_p $ vanishes to order $p$ along $\Si$. Assume
there exists $C$ such that 
\begin{gather} \label{eq:2} 
 \Biggl | e ^{- \tau S (x) } \sum_{p = 0 }^{N} \tau^{p/2} f_p (x)
  \Biggr| \leqslant C \tau ^{-1/2},  
\qquad \forall \, \tau\geqslant 1  , \, \forall \, x \in M. 
\end{gather} 
Then, for any $p = 0, \ldots , N$, $f_p$ vanishes to order $p+1$
 along $\Si$. 
\end{prop}

\begin{proof}  Let us prove the first part. By assumption, $ | f ( x) | \leqslant C' (\varphi(x))^{\ell/2} $.
The function $t \rightarrow e^{-t^2} t^\ell$ is bounded on $\R$. We
obtain estimate (\ref{eq:1}) with $C = C' \sup ( e^{-t^2}  t^\ell) $. 

Let us prove the second part. Assume that Equation (\ref{eq:2}) holds. For any positive integer $j$, introduce the functions $b_j$
$$ b_j (x) = \sum_{p=0}^{N} j^{(p+1)/2}  f_p (x) (\varphi(x) ) ^{- (p+1) /2} $$ 
Applying Equation  (\ref{eq:2}) to $\tau = j / \varphi(x)$, we obtain that the function $b_j $ is bounded. Assume now that $j = 1, \ldots , N+1$. Then viewing the numbers $ f_p (x)(\varphi(x)) ^{- (p+1) /2}$ as the solutions of an invertible linear system of Vandermonde type, we deduce that for any $p = 0, \ldots, N$, the function $x\rightarrow  f_p (x) ( \varphi(x))  ^{- (p+1) /2}$ is bounded. This implies that $f_p$ vanishes to order $p+1$ along $\Si$. 
\end{proof}

\begin{rem} \label{rem:casdicret}
The second part of Proposition \ref{prop:estim_symb} always holds if we only assume that equation (\ref{eq:2}) is satisfied for any $\tau \in \N^*$. The proof is the same except that we define the function $b_j(x)$ for $x \in D$ where $D = \{ x \in M / \; \varphi(x) ^{-1} \in \N \}$. The fact that the function $x\rightarrow  f_p (x) ( \varphi(x))  ^{- (p+1) /2}$ is bounded on $D$ is sufficient to conclude that $f_p$ vanishes to order $p+1$ along $\Si$.  \qed
\end{rem} 

\subsection{The class $\He_0 (E, A)$} \label{sec:four-herm-familly}

Consider a Hermitian line bundle $L \rightarrow M$. Let $E$ be a section of $L$ satisfying
\begin{enumerate}[label=(\arabic{section}.\arabic{subsection}.\roman{*}), ref=(\arabic{section}.\arabic{subsection}.\roman{*})]
\item \label{item:E_positif-zero}
$ |E| \leqslant 1$ and $\Si: = \{ |E| = 1 \}$ is a submanifold of $M$,
\item \label{item:E_hessian}
the restriction of the Hessian of $\varphi = -2 \ln |E| $ to the normal bundle of $\Si$ is non degenerate.
\end{enumerate}
Let $A \rightarrow M $ be an auxiliary Hermitian bundle. 
Consider a family 
\begin{gather} \label{eq:Psi}
\Psi = ( \Psi ( \cdot, k ) \in \Ci (M, L^k \otimes A)) , \; k \in \N^*).
\end{gather}
We say that $\Psi$ belongs to $\He_0 (E, A)$  if  $\Psi \in \bigoinf ( 1) $ and there exists a family $(f_\ell; \; \ell \in \Z)$ of $\Ci (M, A)$ satisfying 
\begin{gather} \label{eq:annulation_coef}
 f_{\ell} = \bigo ( -3 \ell) \text{ along } \Si, \qquad \text{ if }\ell \leqslant 0
\end{gather}
and such that
for any $N>0$, we have 
\begin{gather} \label{eq:defPsi}
   \Psi ( \cdot, k ) =  E^k \sum _{\ell \in \Z \cap [ -N, N/2 ]} k^{-\ell} f_\ell  +  \bigo ( k^{- (N+1)/2} ) .
 \end{gather}
As a first observation, since $|E| < 1$ on $M \setminus \Si$, the restriction of any $\Psi \in \He_0 (E, A)$ to $M \setminus \Si$ is in $\bigo ( k^{-\infty})$. There is no additional control outside of $\Si$. More precisely, we have the following easily  checked lemma.  
\begin{lemme} \label{lem:usefullemme}
Let $U$ be a neighborhood of $\Si$. 
A family $\Psi$ of the form  (\ref{eq:Psi}) belongs to $\He_0 (E, A)$ if and only if it is in $\bigoinf ( 1) $, its restriction to $ M \setminus \Si$ is in $\bigo (k^{-\infty})$ and its restriction to $U$ satisfies Equation (\ref{eq:defPsi})  for any $N >0$, with coefficients $f_\ell \in \Ci ( U , \C)$ satisfying (\ref{eq:annulation_coef}).
\end{lemme}
To understand better the expansion (\ref{eq:defPsi}), let us estimate the growth of each term. 
\begin{lemme} \label{lem:trivial}
We have for any $\ell$ and $f_{\ell} \in \Ci ( M, A)$ satisfying (\ref{eq:annulation_coef}) 
$$ E^k k^{-\ell} f_{\ell} = \begin{cases} \bigo (k^{ \ell /2} ) \text{ if $\ell \leqslant 0$} \\ \bigo ( k^{ - \ell} )  \text{ if $\ell \geqslant 0 .$}
\end{cases} $$
\end{lemme}

\begin{proof} 
First, since $|E| \leqslant 1$ on $M$,  we have $E^k f_{\ell} = \bigo (1)$ which proves the result for $\ell$ non negative. 
Then, on the open set $\{ |E| < 1 \}$, we have that $E^k f_{\ell} = \bigo ( k^{-\infty})$. For any $x $ such that $E(x) = 1$, we can write on a neighborhood $U$ of $x$ that $E = e^{ S} \tau$ where $\tau$ is a unitary frame of $L$ and $S \in \Ci(U) $. The function $\varphi = - 2 \re S $ satisfies conditions \ref{item:positif-zero} and \ref{item:hessian}. So we deduce from Proposition \ref{prop:estim_symb}, Equation  (\ref{eq:1}), that $ E^k f_{\ell} = \bigo (k^{3 \ell /2 })$ for any negative $\ell$, so that $ E^k k^{-\ell} f_{\ell} =  \bigo (k^{ \ell /2} )$.
\end{proof} 

Actually a similar proof shows that $E^k k^{-\ell} f_{\ell}$ is in $\bigo_{\infty} (k^{ \ell /2}) $ if $\ell \leqslant 0$ and in $\bigo_{\infty} ( k^{ - \ell} )$  if $\ell \geqslant 0 .$ Recall that by definition any $\Psi \in \He_0 ( E,A)$ is in $\bigoinf (1)$. So we deduce from Proposition \ref{prop:derivatives-control} that the asymptotic expansion (\ref{eq:defPsi}) also holds for the derivatives of $\Psi$. More precisely, 
for any $N >0$, we have 
$$  \Psi ( \cdot, k ) =  E^k \sum _{\ell \in \Z \cap [ -N, N/2 ]} k^{-\ell} f_\ell  +  \bigoinf ( k^{- (N+1)/2} ) .$$ 

As a consequence of Proposition \ref{prop:Borel_bigohalf}, we have the following lemma.
\begin{lemme} \label{lem:Borel_Fourier}
For any family $( f_\ell, \; \ell \in \Z )$ of $\Ci (M, A)$ satisfying (\ref{eq:annulation_coef}), there exists $\Psi \in \He_0 (E,A)$ such that (\ref{eq:defPsi}) holds for any $N$.
\end{lemme}

 For any integer $m \geqslant 0$, we set 
$$\He_m ( E, A) := \He_0 (E, A) \cap \bigo ( k^{- m/2}). $$ 
This defines a filtration $(\He_m (E, A), \; m \in \N)$. By Proposition \ref{prop:derivatives-control} we have that 
$$ \bigcap_{m \in \N} \He_m (E,A) = \bigoinf (1) \cap \bigo ( k^{- \infty}) = \bigoinf ( k^{-\infty}) $$
The fact that $\Psi$ belong to  $\He_m (E, A)$ for some $m$, can be read on the coefficients $f_\ell$ as follows.  
\begin{lemme} \label{lem:symb_filtration}
A family $\Psi \in \bigoinf (1) $ satisfying (\ref{eq:defPsi}) belongs to $\He_m ( E, A)$ if and only if the coefficients $f_\ell$'s satisfy 
\begin{gather}\label{eq:annul_coef}
 f_{\ell} = \begin{cases} \bigo ( -3 \ell ) \text{ along $\Si$ for any $\ell \leqslant -m$} \\ \bigo( m - 2 \ell ) \text{ along $\Si$ for any $-m \leqslant \ell \leqslant  m/2$}
\end{cases} 
\end{gather}
In particular $\Psi ( \cdot, k ) = \bigo ( k^{-\infty})$ if and only if for any $\ell$, the Taylor expansion of $f_\ell$ vanishes along $\Si$. 
\end{lemme}

An equivalent and sometimes usefull way to state Condition (\ref{eq:annul_coef}) is that for any $\ell \in \Z$,  $f_{\ell} = \bigo ( d_{\ell})$ along $\Si$ with $d_{\ell} \geqslant  \max ( m - 2 \ell, - 3 \ell ) $. 

\begin{proof} 
Arguing as in the proof of Lemma \ref{lem:trivial}, we check that $E^k k^{-\ell} f = \bigo ( k^{-( \ell + d/2)})$ when $f = \bigo ( d)$ along $\Si$. For $d = m- 2\ell $, this leads to $E^k k^{-\ell} f = \bigo ( k^{-m/2})$. 
Consequently, if $\Psi$ has the expansion (\ref{eq:defPsi}) with $N = m$ and $f_{\ell} = \bigo ( m-2 \ell)$ for $-m \leqslant \ell \leqslant  m/2$, then $ \Psi = \bigo ( k^{-m/2})$. We prove the converse by induction on $m$.  Assume that $\Psi \in \He_{m+1}  ( E , A)$ and that $f_{\ell} = \bigo ( m-2 \ell)$ for $-m \leqslant \ell \leqslant  m/2$.  Then considering again (\ref{eq:defPsi}) with $N = m$, we get
$$  E^k \sum_{\ell \in \Z \cap [ -m, m/2 ]}   k^{ (m -2\ell)/2} f_{\ell} = \bigo ( k^{-1/2}) $$
Then by the second assertion of Proposition \ref{prop:estim_symb} and Remark \ref{rem:casdicret}, we conclude that $f_{\ell} = \bigo ( m-2 \ell +1)$ for any $-m \leqslant \ell \leqslant  m/2$.
\end{proof} 

It can be helpful to have in mind the following board.
\begin{gather} \label{eq:tableau_qui_aide}
 \begin{array}{ccccccccccccccccccc} 
& & & k^{5}  &   & k^{4}  & & k^{3} & & k^{2} & & k^{1} & & k^0 & & k^{-1}  & & k^{-2} \\
 1 & & & &   & & & &  &  &  &  & \vline & 0 & \vline &  &   & & \\
k^{-1/2} & & & &  & & & &  &  & \vline &  3 & \vline & 1 & \vline &  &   & & \\
k^{-1}  & & & &  & & & & \vline &  6 & \vline &  4 & \vline & 2 & \vline & 0 & \vline  & & \\ 
k^{-3/2}  & & & &  & & \vline & {9}  & \vline &  7 & \vline &  5 & \vline & 3 & \vline & 1 & \vline & & \\
k^{-2}  & & & &\vline  & 12   & \vline & 10  & \vline &  8 & \vline &  6 & \vline & 4 & \vline & 2 & \vline & 0 & \vline \\
k^{-5/2}  & & \vline & 15   &\vline  & 13   & \vline & 11  & \vline &  9 & \vline &  7 & \vline & 5 & \vline & 3 & \vline & 1 & \vline 
\end{array}
\end{gather}
Here if $m$ is the integer in the column of $k^{\ell}$ and the row of $k^{p}$, we have that $E^k k^{\ell} f = \bigoinf ( k^{p}) $  if and only if $f = \bigo ( m) $ along $\Si$.

For instance, we read that $\Psi \in \He_2 (E,A)$ satisfies 
$$ \Psi ( \cdot, k) = E^k \sum_{\ell = -4 }^{2} k^{-\ell} f_{\ell}   + \bigoinf ( k^{ - 5/2})$$
where the coefficients $f_\ell$, $\ell = -4, \ldots , 2$ vanish respectively to order $ 12, 9, 6, 4, 2, 0, 0$. Compare with the underlined numbers in (\ref{eq:exemple_tableau}).
 
\begin{gather} \label{eq:exemple_tableau} 
 \begin{array}{ccccccccccccccccccc} 
  & & & &   & & & &  &  &  &  & \vline & 0 & \vline &  &   & & \\
 & & & &  & & & &  &  & \vline &  3 & \vline & 1 & \vline &  &   & & \\
  & & & &  & & & & \vline &  \underline{6} & \vline &  \underline{4} & \vline & \underline{2} & \vline & \underline{0} & \vline  & & \\ 
  & & & &  & & \vline & \underline{9}  & \vline &  7 & \vline &  5 & \vline & 3 & \vline & 1 & \vline & & \\
 & & & &\vline  & \underline{12}   & \vline & 10  & \vline &  8 & \vline &  6 & \vline & 4 & \vline & 2 & \vline & \underline{0} & \vline \\
  & & \vline & 15   &\vline  & 13   & \vline & 11  & \vline &  9 & \vline &  7 & \vline & 5 & \vline & 3 & \vline & 1 & \vline 
\end{array}
\end{gather}

There is another useful way of writing the asymptotic expansion (\ref{eq:defPsi}) that we will present now. First by Lemma \ref{lem:usefullemme},  there is no real restriction to restrict $M$ to a neighborhood of $\Si$. So we can assume that $M$ is a neighborhood of the null section of a vector bundle $p : B \rightarrow \Si$. Furthermore, we may assume that the bundle $A$ has the form $p^* A_{\Si}$ for some vector bundle $A_\Si \rightarrow \Si$. 
For any $r \in \N$, we denote by $\mathcal{P}_r ( B)$ the vector bundle over $\Si$ whose fiber at $x$ is the space of polynomial map $B_x \rightarrow \C$ with degree at most $r$. Observe that any section $\si$ of $\mathcal{P}_r (B) $ defines a function from $B$ to $\C$, sending $\xi \in B_x$ into $\si(x) (\xi)$. We say that $\si$ is even (resp. odd) if for any $x \in \Si$, $\si(x)$ is even (resp. odd) as a polynomial map.

\begin{prop} \label{prop:MM}
Let $\Psi \in \bigoinf (1)$. Then $\Psi \in \He_0 (E, A)$ if and only if for any $N$, 
\begin{gather} \label{eq:newdev}
 \Psi (x,\xi, k) = E^{k}(x ,\xi )  \sum_{r = 0 }^{N} k^{-r/2} P_r ( x) (k^{1/2} \xi) + \bigo ( k ^{-(N+1)/2} ) , 
\end{gather}
where for any $r \in \N$, $P_r $ is a section of $\mathcal{P}_{3r} (B) \otimes A_\Si$ which has the same parity of $r$. Furthermore $\Psi \in \He_m (E,A)$ if and only if $P_0 = \ldots = P_{m-1} = 0 $. 
\end{prop}

\begin{proof}
Assume that $\Psi \in \He_m (E,A)$ and write 
$$ \Psi = E^k \sum_{\ell \in \Z \cap [-m , m/2]} k^{-\ell} f_\ell \mod \He_{m+1} ( E, A)  $$
with $ f_{\ell} = \bigo ( m - 2 \ell)$ along $\Si$ for any $\ell$. Linearizing along $\Si$, we have
$$ f_{\ell}(x, \xi)  =  f_{\ell ,m} ( x ) (\xi) + \bigo ( m - 2 \ell +1) \quad \text{ along } \Si. 
$$
with $f_{\ell , m}$ a section of $\mathcal{P}_{m - 2 \ell} (B) \otimes A_\Si$. 
By Proposition \ref{prop:estim_symb}, $E^k k^{-\ell}\bigo ( m - 2 \ell + 1) = \bigo ( k^{-(m+1)/2})$. So 
\begin{xalignat*}{2} 
 \Psi (x, \xi)  =  & E^k(x, \xi)  \sum_{\ell \in \Z \cap [-m , m/2]}  k^{-\ell} f_{\ell,m} (x) ( \xi)  \mod \He_{m+1} ( E, A)   \\ = &
E^k(x, \xi)   k^{-m/2} P_{m} ( x, k^{1/2} \xi) \mod \He_{m+1} ( E, A)  
\end{xalignat*}
where
$$ P_m( x) (\xi)= \sum_{\ell \in \Z \cap [-m , m/2] } f_{\ell,m} ( x ) (\xi) $$
Observe that $m -2 \ell$ has the same parity of $m$. Furthermore, $ -m \leqslant \ell \leqslant m/2$ if and only if $ 0 \leqslant m - 2 \ell \leqslant 3m $. 
Now arguing by induction on $N$, we obtain that any $\Psi \in \He_0$  satisfies (\ref{eq:newdev}) with the convenient $P_r$. The proof of the converse is similar. 
\end{proof}

\subsection{Symbols and easy properties} 

Let us define the symbol of an element of $\He (E,A)$. Denote by $N( \Si) \rightarrow \Si$ the normal bundle of $\Si$ and by $ S^m ( N (\Si)^*)$ the $m$-th symmetric power of the dual of $N(\Si)$. For any section $f$ of $A$ vanishing to order $m$ along $\Si$, the linearization of $f$ is a section of $S^m ( N (\Si)^*) \otimes A $ that we denote by $[f]$. The section $[f]$ is null  if and only if $ f$ vanishes to order $m+1$ along $\Si$.  

For any integer $m$, introduce the space
$$ \Symb_m ( \Si , A) = \bigoplus_{ \ell \in \Z \cap [-m, m/2]} \hb ^\ell \Ci ( \Si, S^{m-2\ell} ( N (\Si)^*) \otimes A) .$$
Here the parameter $\hb$ is formal. For any $\Psi \in \He_m ( E, A)$ satisfying Equation (\ref{eq:defPsi}), we set
$$ \si _m ( \Psi) = \sum_{  \ell \in \Z \cap [-m, m/2]} \hb ^\ell [f_\ell]  \in \Symb_m ( \Si , A).$$ 
and we call $\si _m ( \Psi)$ the {\em symbol} of $\Psi$.  
\begin{rem}
Observe that the sum here corresponds to the $m$-th line in the board (\ref{eq:tableau_qui_aide}).  For instance, the symbol of $\Psi \in \He_2 (E, A)$ is the sum of four terms: $h^{-2} g_{-2} + h^{-1} g_{-1} + h ^0 g_0 + h g_1 $ where $g_{-2}$, $g_{-1}$, $g_0$ and $g_1 $ are respectively of degree 6, 4, 2 and 0. \qed 
\end{rem}

\begin{rem}  
With the notation of proposition \ref{prop:MM}, we have 
$$\si _m ( \Psi )  ( \hb , x, \xi)=   \hb^{\ell/2} P_m ( x) ( \hb^{-1/2} \xi ) .$$
where we have identified the normal bundle of $\Si$ with $B$. \qed
\end{rem} 

By Lemma \ref{lem:symb_filtration}, we have for any $m$ an exact sequence
$$ 0 \rightarrow \He_{m+1} ( E, A) \rightarrow \He_{m} (E,A) \xrightarrow{\si_m} \Symb_m ( \Si, A) \rightarrow 0  $$ 
Let us consider the following elementary operations: 
\begin{itemize} 
\item Multiplication by a function $f \in \Ci (M)$ vanishing to order $N$ along $\Si$: it sends $\He_m (E,A)$ in $\He_{m+N} (E,A)$ and multiplies the symbol by $[f]$. 
\item Multiplication by $k^{-\ell}$ with a non negative $\ell$: it sends $\He_m (E,A)$ in $\He_{m+2 \ell} (E,A)$ and multiplies the symbol by $\hbar^{\ell}$. 
\item Multiplication by $k f$ with $f \in \Ci (M)$ vanishing to second order along $\Si$: it sends $\He_m ( E, A)$ in $\bigo( k^{-m/2})$ but it does not preserve $\He_0 (F,A)$.
\item Multiplication by $k f $ with  $f \in \Ci (M)$ vanishing to third order along $\Si$: it sends $\He_{m}(E,A)$ in $\He_{m+1} ( E,A)$, it multiplies the symbol by $\hb^{-1} [f]$.
\end{itemize}
As a consequence, we have the following important property. 

\begin{lemme} \label{lem:indep_E}
Let  $E$ and $E'$ be two sections of $L$  satisfying both the conditions \ref{item:E_positif-zero} and \ref{item:E_hessian}. Assume that $\{ |E| = 1 \} = \{ |E'| =1 \}$ and $E = E' + \bigo(3)$ along $\Si$. Then we have for any $m$ 
\begin{gather*}  
 \He_m ( E, A) = \He_m ( E' , A). 
\end{gather*} 
Furthermore the symbol maps $\si_m : \He_m (E,A) \rightarrow \Symb_m ( \Si , A) $ and  $\si_m : \He_m (E',A) \rightarrow \Symb_m ( \Si , A) $ are the same. 
\end{lemme}

\begin{proof} 
By Lemma \ref{lem:usefullemme}, we can restrict $M$ to any neighborhood of $\Si$. So we can write $ E' = E \exp ( a)$ where $a \in \Ci (M)$ vanishes to third order along $\Si$. Restricting $M$ again if necessary, we have 
$ |  a | \leqslant \tfrac{1}{2} \varphi$
where $\varphi = - 2 \ln |E|$. 

 By Taylor formula, for any $ z \in \C$ and any $m \in \N$, we have 
$$ \exp z = \sum_{\ell = 0 }^m \frac{z^\ell }{\ell !} + r_m (z) \quad \text{ with }  \quad |r_{m} (z) | \leqslant \frac{|z|^{m+1} }{ (m+1)!} e^{|\re z |} .$$
Consequently
$$ ( E')^k = E^k  \sum_{\ell = 0 }^m \frac{(ka)^\ell }{\ell !} + R_m (\cdot, k) $$
where
$$| R_m ( \cdot, k) | \leqslant |E|^k  e^{ k |\re a |} \frac{|ka|^{m+1} }{ (m+1)!} \leqslant  e^{-k  \varphi /2   }   \frac{|ka|^{m+1} }{ (m+1)!} $$
By Proposition \ref{prop:estim_symb}, $R_m ( \cdot, k )$ is in $\bigo ( k^{-(m+1) /2})$. This proves that the family $((E')^k$, $k \in \N)$ is in $\He_0 ( E, \C)$ and that its symbol is the function constant equal to 1. We easily conclude that $\He_m ( E, A) = \He_m ( E' , A)$ for any $m$ and that the symbol maps are the same.   
\end{proof} 

\subsection{The section $E$} 

Consider a Hermitian line bundle $L \rightarrow M$ with a connection $\nabla$ of curvature $\frac{1}{i } \om$. Let $\Si$ be a closed submanifold of $M$ such that the restriction of $\om$ to $\Si$ vanishes. So the restriction of $\nabla$ to $L \rightarrow \Si$ is flat. Assume that the holomomy of $L \rightarrow \Si$ is trivial. So there exist a non-vanishing flat section $t: \Si  \rightarrow L$. If $\Si$ is connected, it is unique up to multiplication by a complex number.

Extend $t$ to a section $E$ of $L \rightarrow M$ and introduce the one form $\al_E$ defined on a neighborhood of $\Si$ by the equation 
\begin{gather} \label{eq:def_alphaE}
 \nabla E = \frac{1}{i} \al_E \otimes E .
\end{gather}
Assume that $\al_E$ vanishes along $\Si$. Then there exists a section $B_E$ of $(T^*M \otimes T^*M ) \otimes \C \rightarrow \Si$ such that for any vector fields $X$ and $Y$ of $M$ 
\begin{gather} \label{eq:def_BE}
 X. \al_E ( Y) = B_E( X, Y)  
\end{gather}
 along $\Si$. 
This sections encodes the second derivatives of $E$ along $\Si$. It satisfies the following two conditions: for any $p\in \Si$, 
\begin{enumerate}[label=(\arabic{section}.\arabic{subsection}.\roman{*}), ref=(\arabic{section}.\arabic{subsection}.\roman{*})] 
\item  \label{item:1} for any $X \in T_p \Si$ and $Y \in T_p M$, $ B_E( X, Y)  =0$,
\item  \label{item:2} for any $X$, $Y \in T_p M$, $ B_E(X, Y) - B_E( Y, X) = \om (X, Y) $
\end{enumerate}
The following lemma tells us that these are the only conditions $B_E$ have to satisfy.
\begin{lemme} \label{lem:cond_B}
For any section $B$ of   $(T^*M \otimes T^*M ) \otimes \C \rightarrow \Si$ satisfying conditions \ref{item:1} and \ref{item:2}, there exists a section $E$ of $L \rightarrow M$ such that $E |_\Si = t $, the associated one form $\al_E$ vanishes at any point of $\Si$ and $B_E = B$.  This section is unique up to a section vanishing to order 3 along $\Si$. 
\end{lemme} 

\begin{proof} 
Consider any section $E$ whose restriction to $\Si$ is $t$. Let  $E' = e^{if} E$ with $f \in \Ci (M)$. If $f$ vanishes along $\Si$, then the restriction of $E$ to $\Si$ does not change and $\al_{E'} = \al_{E} + df $. So replacing by $E$ by $e^{if} E$ with $f$ conveniently chosen, we have that $\al_E =0 $ along $\Si$. Assume it is the case. If $f$ vanishes to the second order along $\Si$, then $E$ and $E'$ satisfy both $\al_E = 0$, $\al_{E'} =0$ along $\Si$. Furthermore we have that $B_{E'} = B_E + \op{Hess} f$. Again, choosing $f$ conveniently, we obtain that $B_{E'} = B$.
\end{proof} 

Assume now that the curvature $\om$ in symplectic and that $\Si$ is Lagrangian. Consider an almost complex structure $j$ compatible with $\om$. Denote by $T^{1,0} M $ the subbundle $\ker ( \op{id} - ij )$ of $TM \otimes \C$ and by $T^{0,1} M$ its conjugate. 

\begin{prop} \label{prop:section-E}
For any non vanishing flat section $t$ of $L \rightarrow \Si$, there exists a section $E$ of $L \rightarrow M$ such that $E|_ \Si = t$ and for any  $Z \in \Ci ( M , T^{1,0} M)$, $$\nabla_{\con{Z}} E = \bigo (2) \text{ along } \Si. $$ This section is unique up to a section vanishing to order 3 along $\Si$. The corresponding section $B_E$ is given by 
\begin{gather} \label{eq:B_E}
 B_E ( X, Y) = \om ( q(X), Y) , \qquad \forall p \in \Si, \; \forall X, Y \in T_p M\otimes \C \end{gather}
where $q$ is the projection of $T_pM \otimes \C$ onto $T^{0,1}_p M$ with kernel $T_p \Si \otimes \C$. 
\end{prop}

\begin{proof}
Since $\Si$ is Lagrangian, for any $p \in \Si$, $$T^{0,1}_p M \oplus (T_p \Si  \otimes \C) = T_p M \otimes \C.$$
So the fact that $E|_\Si = t $ with $t$ flat and $\nabla_{\con{Z}} E = 0$ along $\Si$ for any  $ Z \in \Ci ( M , T^{1,0} M)$ implies that $\al_E = 0$. If this is satisfied, then the condition $\nabla_{\con{Z}} E = \bigo (2)$ along $\Si$  is equivalent to 
$$ B_E (Y, \overline{Z} ) = 0 , \qquad \forall Y \in T_p M .$$ 
Since the antisymmetric part of $B_E$ is $\om$, we conclude that necessarily $B_E ( \con{Z} , Y ) = \om ( \con{Z}, Y)$. Furthermore $B_E( X, Y) = 0$ for any $X \in T_p \Si$. So necessarily $B_E$ satisfies equation (\ref{eq:B_E}). Assume now that $B_E$ is defined by this formula. Then it satisfies condition \ref{item:1}. Let us check Condition \ref{item:2}. Since $T^{0,1}_pM$ and $T_p \Si$ are Lagrangian, we have for any $X, Y \in T_p M$, 
$$ \om ( q X, q Y) = 0  \qquad \text{ and } \qquad \om ( X - q X, Y - q Y) =0. $$
 This implies that 
$$ \om ( q X, Y) - \om ( q Y, X) = \om  ( X, Y),$$
which was to be proved. The existence of $E$ is now a consequence of Lemma \ref{lem:cond_B}. 
\end{proof}

\begin{rem}  \label{rem:section-E}
Let $E$ be a section of $L$ satisfying the conditions of Proposition \ref{prop:section-E}. Let $\varphi = -2 \ln |E|$. Then $ - d \varphi = \frac{1}{i} ( \al_E + \con{\al}_E)$. So $d\varphi$ vanishes along $\Si$ and 
$$ - \op{Hess} ( \varphi) (X, Y) = \frac{1}{i} ( B_E( X, Y) - \con{ B_E( X, Y)} ), \qquad \forall X, Y \in T_ p M  .$$  
Since $j$ is compatible with $\om$, $T_pM$ admits a natural scalar product $g$ given by 
$$g(X, Y) = \om ( X, j Y).$$ 
Since $\Si$ is Lagrangian, the orthogonal of $T_p \Si$ is $j ( T_p \Si)$. So if $ X $ is orthogonal to $T_p \Si$, then $q(X) = X + ij X$. Then Equation (\ref{eq:B_E}) implies that 
\begin{gather} \label{eq:norm}
 \op{Hess} ( \varphi) (X, Y) =  2 g (X, Y)
\end{gather}
for any $X$ and $Y$ in the orthogonal subspace of $T_pM$. 

Assume now that $|t| =1$ so that $\varphi =0 $ on $\Si$. Since the Hessian of $\varphi$ is positive on the orthogonal of $T\Si$, there exists a neighborhood $U$ of $\Si$ such that $\varphi$ is positive on $U \setminus \Si$. Modifying $E$ outside a neighborhood of $\Si$, Conditions \ref{item:E_positif-zero} and \ref{item:E_hessian} are satisfied. \qed
\end{rem}

Let us present an alternative construction of the section $E$. There is no real restriction to reduce $M$ to a tubular neighborhood of $\Si$, so we can assume that $M$ is an open neighborhood of the null section of a vector bundle $p :B \rightarrow \Si$. Restricting $M$ again if necessary, we can assume that for any $\xi \in B_x \cap M $, the path $\ga_\xi : t \in [0,1] \rightarrow t\xi$ is contained in $M$. Then starting from a flat non vanishing section $t : \Si \rightarrow L$, we can extend $t$ to $M$ by parallel transport along these paths. Observe that $B$ has a natural metric obtained by identifying $B$ with a subbundle of the restriction of $TM$ to $\Si$, so for any $\xi \in B_x$, $ | \xi |^2 = g ( \dot{\ga}_{\xi}(0),  \dot{\ga}_{\xi}(0))$.   
\begin{prop} \label{prop:altern}
The section $E ( x, \xi) = e^{-|\xi | ^2 /4} t(x, \xi ) $ satisfies the conditions of Proposition \ref{prop:section-E}. 
\end{prop} 

\begin{proof}
Clearly the covariant derivative of $t$ vanishes along $\Si$. So the same holds for $E$. 
Let $x \in \Si$ be fixed. Let us compute the second derivative of the section $E( x, \cdot )$ of $L \rightarrow B_x$. Since the curvature of $L$ is $\frac{1}{i} \om$ and the section $t(\cdot, x)$ is flat along the paths $\ga_{\xi}, \xi \in B_x$, we have that 
$  \nabla t(x, \cdot) = \tfrac{1}{i} \al_x \otimes t(x, \cdot)$
with $\al_x $ a 1-form of $B_x$ vanishing at the origin and such that  
$$ \xi_1 . \al_x ( \xi_2) = \tfrac{1}{2} \om ( \xi_1 , \xi_2) + \bigo (1)$$ 
 for any constant vector fields $\xi_1$, $\xi_2$ of $B_x $. Here we identify as above $B_x$ with a subspace of $T_x M$ and the $\bigo (1)$ is a function of $B_x$ vanishing at the origin. Then 
$$ \nabla  E( x, \cdot) = \tfrac{1}{i} \be_x \otimes E( x, \cdot) $$ 
with $\be_x = \al_x - \tfrac{i}{4} d |\eta|^2$. So we have at the origin
$$ \xi_1 .\be_x ( \xi_2) = \tfrac{1}{2} \om ( \xi_1, \xi_2) - \tfrac{i}{2} \om ( \xi_1,j \xi_2) = \om ( \xi_1 , \xi_2 ^{1,0} )  = \om ( \xi_1^{0,1}, \xi_2), $$
where we have used that $T_x ^{0,1}M$ and $T_x^{1,0}M$ are Lagrangian. 
So Equation (\ref{eq:B_E}) is satisfied for $X$ and $Y \in B_x$. Then, by Conditions \ref{item:1} and \ref{item:2}, Equation (\ref{eq:B_E}) is satisfied for any $X, Y \in T_x M$.  In particular $B_E(X,Y)$ vanishes when $ Y \in T^{0,1}_xM$. 
\end{proof} 

\subsection{Further properties of the section $E$} 

In this section we give some properties of the section $E$, that we will use later to compute the commutator of Toeplitz operators. 

Let $E$ be a section of $L$ satisfying the conditions in Proposition \ref{prop:section-E}. Let $f \in \Ci (M, \R)$ and $X$ be the Hamiltonian vector field of $f$. Let $a \in \Ci (M^2)$ be such that on a neighborhood of $\Si$, we have: $ ( f + i \nabla_X ) E = a E$. 

\begin{prop} \label{prop:secE_com}
Assume that $f$ vanishes along $\Si$. Then the function $a$ vanishes to second order along $\Si$. Furthermore, for any sections $U,V$ of $T^{1,0}M$, we have
$$ \con{U}. \con{V}. a = \om ( \con{U}, [ \con{V} , X ])$$
along $\Si$. 
\end{prop}

\begin{proof} 
By definition, $a = f + \al_E ( X)$ where $\al_E$ is defined in (\ref{eq:def_alphaE}). Since $f$ and $\al_E$ vanish along $\Si$, the same holds for $a$. Let us compute the derivative of $a$ with respect to a vector field $Y$.
\begin{xalignat*}{2} 
Y.a = & \om ( X,Y) + Y. \al_E(X) \quad \text{ because } \om(X, \cdot ) = df \\
= & \om (X,Y) + X.\al_E ( Y) + \om ( Y,X) + \al_{E} ( [Y,X]) 
\quad \text{ because } d\al_E = \om \\
= &  X.\al_E ( Y) + \al_{E} ( [Y,X])
\end{xalignat*} 
This vanishes along $\Si$ because $\al_E$ vanishes along $\Si$ and $X$ is tangent to $\Si$. Assume now that $U$, $V$ are sections of $T^{1,0} M$. Putting $Y = \con{V}$ in the last equation and using that $d \al_E = \om$, we obtain
\begin{xalignat*}{2} 
\con{U}.\con{V} .a = &  \con{U}. X.\al_E (\con{V}) + \con{U}.  \al_E ( [\con{V} ,X]) \\ 
= & \con{U}. X.\al_E (\con{V})+ [\con{V}, X] . \al_E ( \con{U}) + \om ( \con{U}, [\con{V}, X]) + \al_E ( [ \con {U}, [ \con{V}, X]) .
\end{xalignat*}
The last term of the right hand side vanishes along $\Si$ because $\al_E$ vanishes along $\Si$.  The second term vanishes too along $\Si$ because $\al_E( \con{U})$ vanishes to second order along $\Si$. Finally, $\al_E (\con{V})$ vanishing to second order along $\Si$ and $X$ being tangent to $\Si$, $X. \al_E ( \con{V})$ vanishes to second order $\Si$. So $  \con{U}. X.\al_E (\con{V})$ vanishes along $\Si$. This concludes the proof. 
\end{proof} 
 
Introduce an auxiliary bundle $A \rightarrow M$ and consider the spaces $\He_{m} (E, A)$ defined in Section \ref{sec:four-herm-familly}. 
Let $X$ be a vector field $X$ of $M$ and $D_X$ be a derivative 
$$ D_X : \Ci ( M, A) \rightarrow \Ci ( M, A)$$ 
in the direction of $X$. Introduce the derivative $P_{X,k}$ of $L^k \otimes A$ given by
$$ P_{X,k} = ( \nabla_X^{L^k} \otimes \op{id} + \op{id} \otimes D_X ) : \Ci (M, L^k \otimes A) \rightarrow \Ci (M, L^k \otimes A).$$

\begin{lemme} \label{lem:derEstim} 
Let $\Psi \in \He_0 ( E, A)$ with symbol $\si_{0}$. Then $( \frac{i}{k} P_{X,k} \Psi_k) \in \He_{1}  (E,A)$. Its symbol is $\tau \si_{0}$ where  $\tau \in  \Ci ( \Si, N (\Si)^* )$ is given by $$\langle \tau, Y\rangle = B_E ( Y, X (p) ) , \qquad \forall p \in \Si,\; Y \in T_p M.$$
\end{lemme}

\begin{proof}
For any section $g \in \Ci ( M , A)$, we have that 
$$ \tfrac{i}{k} P_{X,k} ( E^k g) =  E^k \bigl( \al_E (X)  g +  \tfrac{i}{k} D_X g \bigr)$$
$\al_E(X)$ vanishes along $\Si$ with linear part $\tau$. Clearly if $g = \bigo  (d) $ along $\Si$, then $\al_E (X)  g =\bigo ( d+1)$ along $\Si$ and $D_X g = \bigo ( d-1)$ along $\Si$. We conclude by applying Lemma \ref{lem:symb_filtration}. 
\end{proof}

Let $f \in \Ci (M)$. Assume that the restriction of $f$ to $\Si$ vanishes and that $X$ is the Hamiltonian vector field of $f$.
\begin{lemme} \label{lem:derhamE}
For any  $\Psi \in \He_0 ( E, A) $ with symbol $\si_0$, $( (f + \frac{i}{k} P_{X,k}) \Psi_k ))$ belongs to $\He_{2} (E, A)$ and has symbol 
$$\si_0 [a] + i \hb D_X \si_0 ,$$ where $[a] \in \Ci ( M , S^2 ( N (\Si)^*))$ is the linearization of the function $a$ introduced above.  
\end{lemme} 

\begin{proof}
For any section $g \in \Ci ( M , A)$, we have that
$$ (f + \tfrac{i}{k} P_{X,k})( E^k g )  =  E^k \bigl( a g +  \tfrac{i}{k} D_X g \bigr).$$
Since $a$ vanishes to second order along $\Si$, $g = \bigo ( d ) $ along $\Si$ implies that $ ga = \bigo ( d+2)$ along $\Si$. Since $X$ is tangent to $\Si$, $g = \bigo (  d) $ along $\Si$ implies that $ D_X g = \bigo ( d)$ along $\Si$. We conclude by applying Lemma \ref{lem:symb_filtration}.
\end{proof}

\section{The algebra $\Fa (M, L,j  ,A)$} \label{sec:algebra}

In the first section we introduce several algebras related to Bargmann space. 

\subsection{Algebraic preliminaries}

For any complex finite dimensional space $E$, let $\mathcal{P}(E)$ be 
the space of functions from $E$ to $\C$, whose real and imaginary part are
real polynomials of $E$. Define 
$$\pol (E) [\hbar^{\pm 1}] := \bigoplus_{\ell \in \Z } \hbar^{\ell} \pol[E].$$
The elements of $\pol
(E)[\hbar^{\pm 1}]$ will be considered as maps from $\R_{>0} \times E
\ni ( \hbar, x)$ to $\C$. So far, we have only used that $E$ is a real
vector space. Using the complex structure, we can write any $f \in \pol (E)
[\hbar^{\pm 1}]$ on the form 
\begin{gather} \label{eq:10}
f( \hbar ,  x) = g(\hbar, x, \con{x} )
\end{gather}
 where
$g$ is a function from  $\R_{>0}
\times E \times \con{E} \ni (\hbar, x_1, \con{x}_2)$ to $\C$ which is
complex polynomial in $x_1$ and $\con{x}_2$. $g$ is uniquely determined by
$f$ 

Let $z^1, \ldots, z^n$ be a system of complex linear coordinates of $E$. Then the
family $( \hbar^\ell z^\al \con{z}^\be, \ell \in \Z, \al, \be \in \N^n)$ is
a basis of $\pol (E) [\hbar^{\pm 1}]$. If  $f$ is the monomial $\hbar ^\ell z^\al \con{z}^\be$, then the map $g$ defined in (\ref{eq:10}) is given by $g (\hbar,  x, \con{y}) = \hbar^\ell z^\al(x) \con{z}^\be(y)$.

\subsubsection{The algebra $(\U,\circ)$} 

Let $n \in \N^*$  and $\U = \pol (\C^{n} \times \C^n ) [\hbar^{\pm 1}]$. So $\U$ is the space of applications from $\R_{>0} \times \C^n \times \C^n \ni ( \hbar, z_1, z_2)$ to $\C$, with basis the set of monomials 
$$ \hbar^\ell z_1^{\al_1} \con{z}_1^{\be_1} z_2 ^{\al_2} \con{z}_2^{\be_2} \qquad \ell \in \Z, \; \al_1,\be_1, \al_2 , \be_2 \in \N^n.$$ 
We can write each element of $\U$ on the form $P(\hbar , z_1,\con{z}_1, z_2, \con{z}_2)$ as explained above.
Define the product $\circ$ of $\U$ 
\begin{gather} \label{eq:prodV} 
 ( P \circ Q )(\hbar, z_1, \con{z}_1, z_2, \con{z}_2) = \bigl(\exp ( \hbar \square ) R \bigl) (\hbar, z_1, \con{z}_1, z_1, \con{z}_2, z_2, \con{z}_2) 
\end{gather}
where $R$ and $\square$ are the function and operator given by
\begin{gather} \label{eq:def_R}
R (\hbar, z_1, \con{z}_1, \zeta, \con{\zeta}, z_2, \con{z}_2) = P( \hbar, z_1, \con{z}_1, \zeta, \con{\zeta}) Q(\hbar,\zeta, \con{\zeta}, z_2, \con{z}_2) \\
\label{eq:def_lap_zeta}
   \square = \sum_{i=1}^{n} \frac{\partial^2 }{\partial \zeta^i \partial \con{\zeta}^i} 
\end{gather}
and $\exp ( \hbar \square ) R = \sum_{\ell =0 }^{\infty} \hbar^{\ell} \square ^{\ell} R / \ell!$, the sum being actually finite since $R$ is polynomial in $\zeta_i$, $\con{\zeta}_i$. 

This product corresponds to the composition of some Schwartz kernels as we explain now. Introduce the function $\psi$ of $\C^n \times \C^n$ given by 
\begin{gather} \label{eq:def_psi}
 \psi (z_1, z_2)  = \frac{1}{2} \sum_{i=1}^{n}  \bigl( |z_1^i|^2 + |z_2^i|^2 - 2 z_1^i \con{z}^i_2 \bigr). 
\end{gather}
For any $P \in \U$, define for each $\hbar >0$ the function $K_P ( \hbar, \cdot)$ of $\C^n \times \C^n$  
$$ K_P ( \hbar, z_1, z_2) =  ( 2 \pi \hbar )^{-n} \exp ( - \hbar^{-1} \psi (z_1, z_2) ) P ( \hbar, z_1, \con{z}_1, z_2, \con{z}_2) $$
Introduce the measure $\mu_n= | dz^1 d\con{z}^{1} \ldots dz^n d\con{z}^n|$ of $\C^n$.
\begin{prop} \label{prop:comp_produit}
For any $P,Q \in \U$,  for any $(z_1, z_2 ) \in \C^n\times \C^n$
and $\hbar >0$, we have
\begin{gather} \label{eq:11}
 K_{P \circ Q} (\hbar, z_1, z_2) = \int_{\C^n} K_P ( \hbar, z_1, \zeta) K_Q ( \hbar, \zeta, z_2) \; d\mu_n (\zeta) .
\end{gather}
\end{prop}
\begin{proof} First, the integral converges because the real part of $\psi ( z_1,z_2)$ is $ \frac{1}{2} |z_1 - z_2 |^2$.
One has 
\begin{gather} \label{eq:rel_psi}
\psi ( z_1, \zeta) + \psi ( \zeta, z_2) = \psi( z_1,z_2) + \sum_{i=1}^n ( \zeta^i  - z_1^i  ) (
  \con \zeta^i  - \con z_2^i ) .
\end{gather}
So the right hand side of (\ref{eq:11}) is equal to  
  $$ (2\pi \hbar )^{-2n}  \int_{\C^n} e^{ -  \hbar^{-1}  \sum
(\zeta^i - z_1^i ) ( \con{\zeta}^i - \con{z}_2 ^i ) }R (\hbar, z_1, \con{z}_1,
\zeta, \con{\zeta}, z_2, \con{z}_2)  \; \mu_n ( \zeta ) $$
where $R$ is  the function introduced in (\ref{eq:def_R}). Assume that $P$ and $Q$ are
both monomials so that we can separate variables. We conclude with lemma \ref{lem:int}.
\end{proof}

\begin{lemme} \label{lem:int}
For any $f \in \pol (\C)$, we have for any
$\sigma, \tau \in \C$ that 
$$ (2 \pi \hbar )^{-1} \int_{\C} e^{ - \hbar^{-1} ( \lambda - \sigma ) ( \con{\lambda} -
  \con{ \tau} )} f( \lambda, \con{\lambda}) \; | d \lambda \wedge d \con{\lambda}| =
  \sum_{\ell =0}^{\infty} \frac{ \hbar^\ell}{\ell !} (\Delta^\ell f)  ( \sigma, \con{\tau} )
 $$ 
where $\Delta = \partial_\lambda \partial_{\con{\lambda}}$. 
\end{lemme}

\begin{proof} First, by doing the change of variable $\la' = \sqrt{\hbar} \la$, we see that it is sufficient to prove the formula for $\hbar =1$. So assume $\hbar=1$. Observe that both sides of the formula depends holomorphically on $\si$ and anti-holomorphically on $\tau$. So it is sufficient to prove the formula for $\si = \tau $. If now $\si = \tau$, by doing the change of variable  $\la' = \la - \sigma$, we are reduced to prove the result for $\sigma = \tau = 0 $. Finally, for the monomial $f(\la, \con{\la} ) = \la^n \con{\la}^m$, we have to check 
$$  (2 \pi )^{-1} \int_{\C} e^{ - |\la|^2} \la^n \con{\la}^m  \; | d \lambda \wedge d \con{\lambda}| = \begin{cases} 0 \text{ if $ n\neq m$}\\ n! \text{ if $n =m$},   \end{cases} $$
which can be proved easily in polar coordinates. 
\end{proof}

In the sequel, we will  deduce everything from the definition
(\ref{eq:prodV}) without using Proposition \ref{prop:comp_produit}. Nevertheless,
even if we don't need it, it is certainly helpful to know that $\U$ is an
algebra of operators acting on functions on $\C^n$, the Schwartz kernels of
these operators being the $K_P$'s.  In particular, $K_1$ is the Bergman kernel in disguise. More precisely, let $\Pi$ be the orthogonal projector of $L^2 ( \C^n , e^{- |z|^2})$ onto the subspace of holomorphic functions. Then $K_1 ( \hbar, \cdot)$ is the Schwartz kernel of $U_{\hbar} \Pi U_{\hbar}^{-1}$ where $U_\hbar $ is the map sending a function $f : \C^n \rightarrow \C$ into $g: \C^n \rightarrow \C$ given by $ g(z) = \hbar^n f( \sqrt{\hbar} z) e^{-|z|^2/2}$.
 
Since $\circ$ corresponds to a product of operators, we expect it to
be associative. This can be checked directly from formula (\ref{eq:prodV}), by establishing the following formula for the product of three terms
$$  ( P \circ Q \circ R ) (\hbar, z_1, \con{z}_1, z_2, \con{z}_2) = \bigl(\exp ( \hbar \square_2 ) S \bigl) (\hbar, z_1, \con{z}_1, z_1, \con{z}_2, z_1, \con{z}_2, z_2, \con{z}_2) $$ 
where $\square_2 = \sum_{i=1}^{n} \bigl(  \partial_{\zeta_1^i} \partial_{\con{\zeta}_1^i} + \partial_{\zeta_1^i} \partial_{\con{\zeta}_2^i} +  \partial_{\zeta_2^i} \partial_{\con{\zeta}_2^i} \bigr)$ and 
\begin{gather*} S(\hbar, z_1, \con{z}_1, \zeta_1, \con{\zeta}_1, \zeta_2, \con{\zeta}_2, z_2, \con{z}_2) = \\ P( \hbar, z_1, \con{z}_1, \zeta_1, \con{\zeta}_1) Q(\hbar,\zeta_1, \con{\zeta}_1, \zeta_2, \con{\zeta}_2) S( \zeta_2, \con{\zeta}_2, z_2, \con{z}_2) .
\end{gather*}

\subsubsection{The subalgebra $\V$}
Let $\V$ be the subspace of $\U$ generated by the monomials $\hbar^\ell z_1^{\al} \con{z}_2^{\be}$, with $\ell \in \Z$, $\al, \be \in \N^n$.  In other words, $\V$ consists of the elements of $\U$ which are independent of $\con{z}_1$ and $z_2$. 
\begin{lemme}  $\V$ is a subalgebra of $( \U, \circ) $. Furthermore, for any $P \in \U$, 
\begin{gather} \label{eq:subalg} 
  P \in \U \Leftrightarrow 1 \circ P \circ 1 = P \Leftrightarrow 1 \circ P = P = P \circ 1 .   
\end{gather}
\end{lemme}
\begin{proof} 
Using that $\circ$ is associative and $1$ is idempotent, one checks that the second and third assertions of (\ref{eq:subalg}) are equivalent. The set of $P$ satisfying the third assertion is clearly closed under $\circ$. Now if $1\circ P = P$ (resp. $P \circ 1 = P$), then $P$ does not depend on $\con{z}_1$ (resp. $z_2$) by  (\ref{eq:prodV}). Conversely, if $P$ does not depend on $\con{z}_1$ and $z_2$, one easily see that $1 \circ P = P = P \circ 1$.
\end{proof}
Identifying $\V$ with  $\Vi = \pol ( \C^n) [\hbar^{\pm 1}]$ through the isomorphism
 $$\Vi \rightarrow \V, \qquad f( \hbar, z, \con{z}) \rightarrow f ( \hbar, z_1, \con{z}_2 )$$
the product $\circ$ is sent to the anti-Wick product 
$$  ( f \star_{\op{AW}} g ) ( \hbar, z, \con{z}) = \bigr[ \exp ( \hbar \square) \bigl(  f( \hbar, z,  \con{\zeta}) g( \hbar, \zeta,  \con{z})  \bigr) \bigl]_{ \zeta= z, \; \con{\zeta} = \con{z}}$$ 
where $\square$ is the Laplacian defined in (\ref{eq:def_lap_zeta}). This algebra has a natural representation by differential operators of $\Ci ( \C^n, \C)$: one sends the monomial $\hbar^\ell z^\al \con{z}^\be$ to the operator $\hbar^\ell z^\al D^\be$ where $D^i = \hbar \partial / \partial z^i$. This representation is of course related to the Schwartz kernels introduced above. We won't develop more the properties of $\V$, since we are actually interested in another  subalgebra of $\U$.

\subsubsection{The subalgebra $\W$} \label{sec:subalgebra-w}
Let $\W$ be the subspace of $\U$ generated by the elements $\hbar^\ell (z_2- z_1) ^{\al}(\con{z}_2- \con{z}_1)^{\be}$ with  $\ell \in \Z, \; \al, \be \in \N^n$. We easily see from (\ref{eq:prodV}) that $\W$ is a subalgebra of $\U$. Furthermore, identifying $ \W$ with  $\Vi = \pol ( \C^n) [\hbar^{\pm 1}]$ through the map  
\begin{gather*} 
 \Vi \rightarrow \W, \qquad f( \hbar, z, \con{z}) \rightarrow P ( \hbar, z_1, \con{z}_1, z_2, \con{z}_2 ) = f ( \hbar, z_2 - z_1, \con{z}_2 -\con{z}_1)
\end{gather*}
the product $\circ$   is sent to the product $\star$ of $\Vi$ defined by 
\begin{gather} \label{eq:prodfor} 
 (f \star g ) (\hb, z, \con{z} ) =  \Bigr[  \exp ( \hb \square ) \bigl( f  ( \hb, -\zeta, \con{z} - \con{\zeta} ) g ( \hb , z + \zeta ,  \con{\zeta} ) \bigr) \Bigl]_{\zeta = \con{\zeta} = 0 } 
\end{gather}
where  $\square$ is defined in (\ref{eq:def_lap_zeta}).

For any $m \in \N$, introduce the subspace $\Wi_m$ of $\Vi$ generated by the monomials $\hbar^\ell z^\al \con{z}^\be$ such that 
$$ \al, \be \in \N^n, \quad \ell \in \Z, \quad |\al| + |\be| + 2 \ell = m \quad \text{and} \quad  |\al| + |\be | + 3 \ell \geqslant 0. $$ 
Observe that $\Wi_m $ is finite dimensional.  
\begin{lemme} For any $m_1, m_2 \in \N$,  $ \Wi_{m_1}  \circ \Wi_{m_2} \subset \Wi_{m_1 + m_2 }$.
\end{lemme}
\begin{proof} 
First the product $\star$ is homogeneous with respect to the degree $p ( \hbar^{ \ell} z^\al \con{z}^{\be} ) = |\al | + |\be| + 2\ell$. This follows  from the fact that in Formula (\ref{eq:prodfor}), the Laplacian lowers the degree by 2 whereas the multiplication by $\hbar$ raises the degree by $2$, so that the degree is preserved. Similarly, consider the degree $q ( \hbar^{ \ell} z^\al \con{z}^{\be} ) = |\al | + |\be| + 3\ell $. Since $\hbar \square$ raises the $q$-degree by $1$,  the product of two monomials with non negative $q$-degrees is a linear combination  of monomials with non negative $q$-degree.
\end{proof}



\begin{lemme} 
For any $P \in \W_m$, $K_P ( \hbar, \cdot)$ is the Schwartz kernel of a bounded operator of $L^2 ( \C^n)$ with uniform norm in $\bigo ( \hbar ^{m/2})$. 
\end{lemme}
\begin{proof}  By Schur criterion, it is sufficient to prove that there exists $C>0$ such that 
 for any $\hbar >0$ and $z_1, z_2 \in \C^n$, one has   
\begin{gather} \label{eq:7}
 \int_{\C^n} | K_P ( \hbar, z_1 , z) |  \; \mu_n ( z) \leqslant C \hbar ^{m/2}, \quad \int_{\C^n} | K_P ( \hbar, z , z_2) |  \; \mu_n ( z) \leqslant C \hbar ^{m/2}.
\end{gather}
Assume that $f( \hbar, z, \con{z} ) = \hbar^{\ell} z^\al \con{z}^{\be}$ with $|\al| + | \be | + 2 \ell = m$. Then 
$$  | K_P ( \hbar, z_1 , z_2) | \leqslant ( 2\pi \hb )^{-n} e^{ - \frac{1}{2 \hbar} |z_2 - z_1 | ^2 } \hbar^{\ell} |z_1 - z_2 | ^{m - 2\ell}$$ 
Consequently
\begin{gather*} 
 \int_{\C^n} | K_P ( \hbar, z_1 , z) |  \; \mu_n ( z)  = \int_{\C^n} | K_P ( \hbar, z_1 , z_1+z) |  \; \mu_n ( z)  \\ 
\leqslant (2\pi \hb)^{-n} \int_{\C^n}  e^{ - \frac{1}{2 \hbar} |z| ^2 } \hbar^{\ell} |z| ^{m - 2\ell}   \; \mu_n ( z)  = \frac{ \hbar^{m/2}}{(2\pi)^{n}} \int _{\C^n}  e^{ - \frac{1}{2} |z| ^2 }  |z| ^{m - 2\ell}   \; \mu_n ( z)
\end{gather*}
which proves the first estimate of (\ref{eq:7}), the second is similar. 
\end{proof}

As a last observation, we can define the product $\star$ on $\Wi (E) = \C [ E] [\hb^{\pm 1}]$ when $E$ is any finite dimensional vector space endowed with a symplectic and a complex structure $\om,j$. To do this, we choose a complex basis $(e_j)$ of $E$ such that $\om ( e_j,  e_k ) = i \delta_{jk}$, consider the  associated linear coordinates of $E$ and define the Laplacian (\ref{eq:def_lap_zeta}) by using these coordinates. The resulting operator does not depend on the choice of $(e_j)$.    

Introduce a finite dimensional complex vector space $A$. Let $ \Wi (E, A)$ be the algebra $ \Wi ( E) \otimes \op{End} A $. We still denote the product by $\star$. We also have that $$\Wi_{m_1} (E,A) \star \Wi_{m_2} (E,A) \subset \Wi_{m_1 + m_2 } (E,A)$$
 where $\Wi_m (E,A) := \Wi_m (E) \otimes \op{End} (A)$. 
Furthermore $1_A = 1 \otimes \op{id}_A$ is an idempotent of $\Wi (E,A)$. 
In the sequel we will need the following lemma
\begin{lemme} \label{lem:ex_proj_symb} For any integer $m \geqslant 0$ and for any $f \in  \Wi_m ( E, A)$, we have the following equivalence
$$ f \star 1_A = f= 1_A \star f \Leftrightarrow f ( \hb, z, \con {z} ) = f ( \hb, 0 ,0) .$$
\end{lemme}
\begin{proof} This follows from the formulas $ (1_A \star f ) ( \hb, z, \con{z} ) =  \bigl[\exp ( - \hb \Delta ) f \bigr] (\hb,  z , 0 )$ and 
$( f \star 1_A ) ( \hb ,  z, \con{z} ) =  \bigl[\exp ( - \hb \Delta ) f \bigr]  ( \hb, 0, \con{z} )$.
\end{proof}

\subsection{The algebra $\Fa$} \label{sec:defin-fa}

Let $M$ be a symplectic manifold. Let $F \rightarrow M$ be a Hermitian vector bundle. We denote by $F \boxtimes \con{F} \rightarrow M ^2$ the vector bundle $p_1^* F \otimes p_2^* \con{F}$ where $p_1$ and $p_2$ are the projections $M^2 \rightarrow M$ onto the first and second factor respectively. 
We use the following convention for Schwartz kernel: if  $K$ is a section of $\Ci ( M ^2, F \boxtimes \overline{F})$, then the operator corresponding to $K$ is the endomorphism $P$ of $\Ci (M , F)$ given by 
\begin{gather} \label{eq:noyauPK}
 (P \Psi ) (x) = \int _M  K ( x, y) \cdot \Psi ( y) \mu (y) 
\end{gather}
where the dot stands for the contraction $ \con{F}_y \otimes {F}_y \rightarrow \C$ induced by the metric and $\mu$ is the Liouville measure. In the sequel we denote the operator $P$ and its Schwartz kernel $K$ by the same letter.

The data to define the algebra $\Fa$ are a compact symplectic manifold $M$, a prequantum bundle $L \rightarrow M$, a Hermitian vector bundle $A \rightarrow M$ and an almost complex structure $j$ of $M$ compatible with $\om$. 

\begin{lemme} \label{lem:sectionE}
There exists a section $E$ of $L \boxtimes \con{L}$ satisfying the following conditions
\begin{enumerate}[label=(\arabic{section}.\arabic{subsection}.\roman{*}), ref=(\arabic{section}.\arabic{subsection}.\roman{*})]
\item \label{item:3} $E(x,x) = 1$, for any $x \in M$,
\item \label{item:4} $\nabla_{(\overline{Z}, 0)} E = \nabla_{(0, Z)} E = \bigo (2)$ along the diagonal of $M^2$, for any vector field $Z \in \Ci ( M , T^{1,0}M)$.
\item \label{item:5} $|E(x, y) | < 1 $ for any $ x \neq y$. 
\item \label{item:6} $ E( y, x ) = \con{E}  ( x, y)$, for any $x, y \in M^2$. 
\end{enumerate}
Such a section is unique  up to a section vanishing to order 3 along the diagonal.
\end{lemme} 
In condition \ref{item:3}, we have identified $L_x \otimes \con{L}_x$ with $\C$ by using the metric of $L$. In condition \ref{item:4}, we have used the connection of $L \boxtimes \con{L}$ induced by the connection of $L$. 

\begin{proof} This is a consequence of  Proposition \ref{prop:section-E}. Indeed, let us endow $M^2$ with the symplectic form $\Om = p_1^* \om - p_2^* \om$, where $p_1$ and $p_2$ are the projection $M^2 \rightarrow M$ onto the first and second factor respectively. Then $\frac{1}{i} \Om$ is the curvature of $L \boxtimes \con{L}$. The diagonal  is a Lagrangian submanifold of $M^2$. Furthermore 
$$(L \boxtimes \con{L}  ) \bigr|_{\op{diag} M} \simeq L \otimes \con{L} \simeq \C, $$
 and the section $s \in \Ci ( \op{diag} M, L \boxtimes \con{L} )$ given by $s (p,p) =1$ is flat. 
So by Proposition \ref{prop:section-E}, we obtain a section satisfying conditions \ref{item:3} and \ref{item:4}. By Remark \ref{rem:section-E}, modifying $E$ outside the diagonal, condition \ref{item:5} is satisfied. Finally observe that $(x, y) \rightarrow \con{E} ( y,x)$ satisfies the same conditions. So we can replace $E(x,y)$ by $\frac{1}{2}( E(x,y) + \con{E} ( y, x))$. And this defines a section satisfying all the requirements.  
\end{proof}

For any $m \in \N$, we define $\Fa_m = \Fa_m (M, L , j  , A)$ as the set consisting of families 
\begin{gather} \label{eq:fam_op}
  P_k : \Ci (M , L^k \otimes A ) \rightarrow \Ci ( M , L^k \otimes A), \qquad k \in \N^* \end{gather}
of linear maps with a smooth Schwartz kernel, 
\begin{gather} \label{eq:schwartz_k}
 P_k ( \cdot, \cdot)  \in \Ci ( M^2, ( L \boxtimes \con{L})^k \otimes ( A \boxtimes \con{A})), \qquad k \in \N^* 
\end{gather} 
such that the family $( P_k ( \cdot, \cdot) ) $ belongs to $k^n \He_m (E, A \boxtimes \overline{A})$. In other words, $(P_k)$ belongs to $\Fa_m$ if the corresponding family of Schwartz kernel is in $\bigoinf ( k^{n})$ and there exists a family $(f_\ell \in \Ci ( M^2, A \boxtimes \con{A}), \; \ell \in \Z)$ such that for any $ N \geqslant 0$, we have 
\begin{gather} \label{eq:noyauFdev}
 P_k  = \Bigl( \frac{k}{2 \pi }\Bigr) ^n E^k   \sum _{\ell \in \Z \cap [ -N, N/2 ]} k^{-\ell} f_\ell   +  \bigo ( k^{n- (N+1)/2} ) . 
\end{gather}
and for any $\ell \leqslant m/2$, we have  
\begin{gather} \label{eq:cond_ann_coef}
 f_{\ell} = \begin{cases} \bigo ( -3 \ell ) \text{ along the diagonal if $\ell \leqslant -m$,} \\ \bigo( m - 2 \ell ) \text{ along the diagonal if $-m \leqslant \ell \leqslant  m/2$.}
\end{cases} 
\end{gather}
 Since the section $E$ is uniquely determined by the complex structure $j$ up to a section vanishing to order $3$ along the diagonal, the set $\Fa_m$ only depends on $j$ and not on the choice of $E$, cf. Lemma \ref{lem:indep_E}.

\begin{theo} \label{theo:main1new} 
The space $\Fa_0$ is closed by multiplication and by adjunction. Furthermore, if $(P_k) \in \Fa_{m}$ and $( Q_k ) \in \Fa_{p}$, then $( P_k ^*) \in \Fa_m $, $(P_k Q_k ) \in \Fa_{m+ p } $.
\end{theo}

\begin{rem} \label{rem:rapport_algebre} 
Let us explain the relation with the algebra $\Wi_0$ introduced in Section \ref{sec:subalgebra-w}. Consider the symplectic form on $\C^n$ given by $ \om = i \sum_i dz^i \wedge d \con{z}^i$. Let $L$ be the trivial Hermitian line bundle on $\C^n$ with connection form $\frac{i}{2} \sum_i (z^i d \con{z}^i - \con{z}^i d z^i )$. Denote by $t$ the canonical frame of $L$ and recall the function $\psi$ defined in (\ref{eq:def_psi}). Then the section $E$ of $L \boxtimes \con{L}$ given by 
\begin{gather} \label{eq:Elinear}
E (x,y) = e^{ -\psi ( x,y)} t(x)  \otimes \con{t}(y)
\end{gather}
satisfies the assumptions of Lemma \ref{lem:sectionE}. Recall that to each $f \in \Wi_0$ is associated the operator with Schwartz kernel 
$$  K_P  ( \hbar, z_1, z_2) = e^{- \hbar^{-1} \psi (x, y) } f( \hbar, z_2 - z_1, \con{z}_2 - \con{z}_1) .$$
Setting $\hbar = k^{-1}$, we recover (\ref{eq:noyauFdev}). Furthermore, when $f \in \Wi_m$, the coefficients $f_\ell$ of the expanion $f (\hbar, z_2 -z_1 ,\con{z}_2 -\con{z}_1) = \sum \hbar^\ell f_{\ell}( z_1, z_2)$ satisfies condition  (\ref{eq:cond_ann_coef}). So except the fact that $\C^n$ is not compact, $\Wi_m$ may be considered as a particular subalgebra of $\Fa_m (\C^n)$ whose elements have polynomial symbol. 
 \qed \end{rem}

\subsection{Proof of Theorem \ref{theo:main1new}}

Let us denote by $A_k $ the bundle $L^k \otimes A$ and by $\pi$ the projection from $M^3$ to $M^2$ given by $\pi (x,y,z) = (x,z)$. 
Consider two families $(R_k)$ and $(S_k)$ consisting of smooth sections
$$ R_k , S_k \in \Ci ( M^2 , A_ k \boxtimes \con{A}_k ), \qquad k \in \N^*. $$
Then define  $T_k \in \Ci (M^3, \pi^* ( A_k \boxtimes \con{A}_k) )$ by 
\begin{gather} \label{eq:def_T} 
 T_k ( x, y, z) := R_k ( x,y) \cdot  S_k  (y,z)
\end{gather} 
where the dot stands for the contraction $A_{k,y} \otimes \con{A}_{k,y}
\rightarrow \C$.   Define also 
\begin{gather}  \label{eq:det_U}
 U_k ( x, z) = \int_M T_k ( x, y, z) \; \mu ( y) , \qquad (x,z) \in M^2 .
\end{gather} 
So $U_k$ is a smooth section of $A_k \boxtimes \con{A}_k$.
Let us start by an easy observation.
\begin{lemme} \label{lem:est1}
If $(R_k) \in \bigoinf ( k^{-m} ) $ and $( S_k) \in \bigoinf ( k^{-p} )$, then $(T_k) \in \bigoinf ( k^{-(m + p) } )$. If $(T_k) \in \bigoinf ( k^{-m})$, then $(U_k) \in\bigoinf ( k^{-m})$.
\end{lemme} 

Let $E$ be a section of $L \boxtimes \con L$ satisfying the Conditions \ref{item:3}, \ref{item:4} and \ref{item:5}. Introduce the section $F$ of $p^*( L \boxtimes \con L)$  given by 
\begin{gather} \label{eq:def_F}
 F(x,y,z) = E(x , y) \cdot E( y, z) 
\end{gather} 
The function $\psi : = - 2 \ln |F|$ satisfies 
$$\psi  ( x,y, z) = \varphi ( x, y) + \varphi ( y, z)$$
 where $\varphi  = - 2 \ln |E|$. Since $\varphi$ satisfies conditions \ref{item:positif-zero} and \ref{item:hessian}, $\psi$ satisfies the same conditions. So the space $\He_p ( F , \pi^* ( A \boxtimes \con A ))$ is well defined. Observe also that the zero level set of $\psi$ is 
$$\Si = \{ (x,x,x) \in M^3/ \; x \in M \}.$$
Theorem \ref{theo:main1new} is a consequence of the following two facts:
\begin{gather} 
\left. \begin{array}{ll} (R_k) \in \He_m ( E, A \boxtimes \con{A}  ) \\ 
\text{and } (S_k) \in \He_p  ( E,
A \boxtimes \con{A}  ) 
\end{array} \right\}
\quad \Rightarrow \quad (T_k) \in \He_{m+ p } (
F, \pi^* ( A \boxtimes \con{A} ) ) \\ \label{eq:second_equation}
(T_k)  \in  \He_{m } ( F, \pi^* ( A \boxtimes \con{A} ) ) \quad \Rightarrow \quad ( k^{ n} U_k ) \in \He_m ( E, A \boxtimes \con{A}  )
\end{gather}
The first implication is easily proved. Let us concentrate on the second one. 
To understand better the strategy of the proof, we discuss in the following remark the composition of Proposition  \ref{prop:comp_produit}. 
\begin{rem}  As explained in remark \ref{rem:rapport_algebre}, the algebra $\Wi_m$ is similar to the algebra $\Fa_m$ with the section $E$ given by (\ref{eq:Elinear}) and polynomial symbol. It follows from (\ref{eq:rel_psi}) that the section $F$ defined by (\ref{eq:def_F}) has the particularly simple expression 
\begin{gather} \label{eq:relation_ideale} 
 F(x,y,z) = e^{ -\sum_i ( z^i( y) - z^i ( x) ) ( \con{z}^i ( y) -
   \con{z}^i ( z) ) } E( x,z) 
\end{gather} 
Since the terms in the exponential are quadratic in the complex variable, the computation in the proof of Proposition  \ref{prop:comp_produit} was relatively easy. In the general case of $\Fa_m$, we can not hope to find complex coordinates $z^i$ such that (\ref{eq:relation_ideale}) holds.  Nevertheless, we will see that we can define some coordinates similar to the following functions 
\begin{gather} u_i (x,y) = z^i (x) - z^i(y), \qquad v_i(x,y,z) = z^i ( x) - \con{z}^i ( z),  \\  w_i ( x,y,z) = z^i ( x) +
\con{z}^i ( z) - 2 z^i (y).
\end{gather}
With these functions, Equation   (\ref{eq:relation_ideale}) reads
 $F = e^{-\frac{1}{4} \sum ( w_i+ v_i ) ( \con{w}_i - \con{v}_i ) } \pi^* E . $
\qed 
\end{rem}

\subsubsection*{Local data on $M^2$} 

Choose $p_0 \in M$. Introduce a local frame $( \partial_i, \; i = 1, \ldots, n )$ of $T^{1,0}M$ on a neighborhood $U$ of $p_0$, which satisfies
$$ \om ( \partial_j, \con{\partial}_k ) = i \delta_{jk} .$$
Consider functions $u_i \in \Ci (U \times U)$ vanishing  along the diagonal and such that 
\begin{gather} \label{eq:coordonnee_u_1}
 du_i ( \partial_j, 0 ) = \delta_{ij} + \bigo ( 1)   , \qquad du_i  ( \con{\partial} _ j , 0 ) = \bigo ( 1)
\end{gather}
Here and in the sequel, we denote by $\bigo ( m)$ any function or section on $U^2$ vanishing to order $m$ along the diagonal of $U^2$. 
Restricting $U$ to a convenient neighborhood of $p_0$, we introduce a coordinate system $(x_i)$ on $U$ centered at $p_0$. Let us extend the functions $x_i$ to $U^2$ by setting $x_i ( x, y) := x_i ( x)$.

\begin{lemme} \label{lem:etap-1}
Restricting $U$ if necessary, the functions 
$$ x_i, \; i =1 , \ldots , 2n, \qquad \re u_j, \; \im u_j, \; j = 1, \ldots, n $$
form a coordinate system on $U^2$.   Furthermore, 
\begin{gather} \label{eq:der_F} 
 \nabla_{ ( \partial_i, 0 ) }  E = - \con{u}_i E + \bigo (2), \qquad \nabla_{ ( \con{\partial}_i, 0 ) }  E =  \bigo (2),   \\ \notag 
\nabla_{ (0, \partial_i) }    E =  \bigo (2), \qquad \nabla_{ (0,  \con{\partial}_i ) }  E = u_i E + \bigo (2).
\end{gather}   
and $|E|^2 = e ^{- |u|^2} + \bigo ( 3) $. The diagonal of $U^2$ is $\{ u = 0 \}$. 
\end{lemme} 

\begin{proof} 
Equation (\ref{eq:coordonnee_u_1}) and the fact that the $u_j$ vanish along the diagonal imply that 
\begin{gather} \label{eq:coordonnee_u_2} 
 du_i ( 0  , \partial_j ) = - \delta_{ij} + \bigo ( 1)   , \qquad du_i  ( 0, \con{\partial} _ j  ) = \bigo ( 1) 
\end{gather} 
We easily deduce that $(x_i, \re u_j, \im u_j)$ is a coordinate system on $U^2$ if $U$ is a sufficiently small neighborhood of $p_0$.

 Write $\nabla E = \frac{1}{i} \al_E \otimes E$ and let $B_E ( X,Y) = X. \al_E ( Y)$ for any $X,Y \in T_{(p,p)} M^2$. As a consequence of equation (\ref{eq:B_E}), we have 
\begin{gather} \label{eq:derEop}
  B_E( (X_1,Y_1) , ( X_2, Y_2)) = \om ( X_1^{0,1} - Y_1 ^{0,1} , X_2) + \om ( X_1^{1,0} - Y_1^{1,0} , Y_2) .
\end{gather}
for any tangent vectors $X_1,Y_1, X_2, Y_2 \in T_p M$.
Here we denote by $X^{1,0}$ and $X^{0,1}$ the components of $X \in T_p M \otimes \C = T^{1,0} _p M \oplus T_p ^{0,1} M$. 

In particular, we have that $( \partial_j, 0). \al_E( \partial_k, 0) = 0 $ and $(\con{\partial}_j , 0 ). \al_E( \partial_k, 0 ) = - i \delta_{jk} $ along the diagonal. We deduce the first equation of (\ref{eq:der_F}) from Equations (\ref{eq:coordonnee_u_1}). The proof of the other equations in  (\ref{eq:der_F}) is similar.

To prove that $$|E|^2 = e^{-|u|^2} + \bigo(3) ,$$ it is sufficient to check that the derivatives of the two sides with respect to $( \partial_i, 0 )$, $(\bar{\partial}_i, 0 )$ are equal up to $\bigo(2)$. This follows from the previous equations. We could also deduce it from Equation (\ref{eq:norm}). 
\end{proof}

\subsubsection*{Local data on $M^3$} 

Introduce the functions $v_i , w_i \in \Ci (U^3)$ given by 
$$ v_ i (x,y,z) = u_i ( x,z) , \qquad w_i ( x, y, z ) =  u_i ( x, y) - u _i ( y, z)  $$
Extend the functions $x_i$ to $U^3$ by setting $x_i ( x, y, z) := x_i ( x, z)$. 
Here and in the sequel, we denote by $\bigo_{\Si} ( m)$ any function or section on $U^3$ vanishing to order $m$ along $\Si$. 
\begin{lemme} \label{lem:etap0}
Restricting $U$ if necessary, the functions 
$$x_i, \; i =1 , \ldots , 2n, \qquad \re v_j,\; \im  v_j, \;  \re w_j, \; \im w_j, \;  j =1 , \ldots , n$$ form a coordinate system on $U^3$.  
We have 
\begin{gather} \label{eq:E_loc}
 F =  e ^{ -\frac{1}{4}  ( w +  v ) \cdot ( \con{w} -  \con{ v} ) } \pi^* E
 + \bigo_\Si( 3) , \qquad |F| ^2 = e^{- \frac{1}{2} (|v|^2 + |w|^2) } + \bigo_{\Si} ( 3).
\end{gather} 
Furthermore $\Si$ intersects $U^3$ in $\{ v = w = 0 \}$. 
\end{lemme} 

\begin{proof} A simple computation from Equations (\ref{eq:coordonnee_u_1}) and (\ref{eq:coordonnee_u_2}) gives 
\begin{gather} \label{eq:der_w}
 d w_j ( 0, \partial_k, 0 ) = -2 \delta_{jk} + \bigo_{\Si} ( 3) , \qquad d w_j ( 0, \con{\partial}_k , 0 ) = \bigo_{\Si} ( 3). \end{gather}
Using that $x_i = \pi^* x_i$, $v_j = \pi^* u_j$ and $(x_i , \re u_j, \im u_j)$ is a coordinate system on $U^2$, we obtain that $( x_i , \re v_j, \im v_j, \re w_j, \im w_j)$ is a coordinate system on $U^3$. 

We have  
\begin{gather} \label{eq:EEE}
 F (x,y,z) =  e ^{ u(x,y) \cdot \con{u} ( y,z )  }  E ( x, z)  + \bigo_\Si( 3) 
\end{gather}
To check that, it is sufficient to prove that the covariant derivative with respect to $( \partial_i, 0 , 0 )$, $ (\con {\partial}_i , 0, 0)$, $( 0,\partial_i,  0 ) $ and $(0, \con{\partial}_i , 0 ) $ of the two sides are equal modulo $\bigo_{\Si} ( 2)$. This follows from a simple computation using Equation (\ref{eq:der_F}) and Equations (\ref{eq:coordonnee_u_1}), (\ref{eq:coordonnee_u_2}). To deduce the first part of Equation (\ref{eq:E_loc}) from (\ref{eq:EEE}),  observe that 
\begin{equation} \label{eq:rel_u_v_w} 
\begin{split}    
u_i ( x, y ) = \frac{1}{2} \bigl( w_i ( x,y,z) +  v_i ( x,y,z) \bigr) +
  \bigo _ \Si (2), \\   
u_i ( y , z ) = \frac{1}{2} \bigl( - w_i (
  x,y,z) + v_i ( x,y,z) \bigr)  + \bigo _ \Si (2)
\end{split}
\end{equation}
which follows again from Equations (\ref{eq:coordonnee_u_1}) and (\ref{eq:coordonnee_u_2}). The computation of the norm $|F|^2$ is similar. 
\end{proof} 

In the sequel we use the coordinates $(x_i, v_j, w_j)$ to compute the integral (\ref{eq:det_U}). Since we integrate with respect to $y$ and $x_i, v_j$ only depend on $x$ and $z$, these coordinates are just parameters and we can use the $w_j$ as the integration variables. To justify our estimates, we will need to specify the value of our coordinates. 

Choose $\epsilon >0$ and an open neighborhood $O$ of $(p_0,p_0,p_0)$ contained in $U^3$ such that the map from $O$ to $ \R^{2n} \times \C^{n} \times \C^{n}$ sending $(x,y,z)$ into $( ( x_i(x,y,z))$, $(v_j (x,y,z)), (w_j ( x,y,z)))$ is a diffeomorphism onto $B_\epsilon ^3$ where $B_{\epsilon}$ is the open ball of $\R^{2n} \simeq \C^n$ centered at the origin with radius $\epsilon$. 

Let $(x,z) \in M^2$ and $O_{x,z}$ be the open set of $M$ defined by  
$$ O_{x,z} = \{ y \in M ;\; (x,y,z) \in O \} . $$ 
Assume $O_{(x,z)}$ is not empty. Then the map from $O_{(x,z)}$ to $\C^n$ sending $y$ into $(w_j ( x,y,z))$ is a diffeomorphism onto $B_{\epsilon}$. Let $\mu_{x,z}$ be the pull-back of the Lebesque measure of $\C^n$ to $O_{x,z}$. Let $\delta \in \Ci ( O^3)$ be the function such that we have on $O_{x,z}$ 
\begin{gather} \label{eq:def_muxz}
 \mu = \delta(x,\cdot, z)  \mu_{x,z}.
\end{gather}
By equation (\ref{eq:der_w}),  $\delta = 1 + \bigo_{\Si}  (1).$
\subsubsection*{Proof of implication (\ref{eq:second_equation})} 

Let us start by several reductions. 
\begin{itemize}
\item By Lemma \ref{lem:est1}, we can treat separately each term in the asymptotic expansion (\ref{eq:noyauFdev}) of $T_k$, so we
assume that $T_k  = F^k g$ where $g \in \Ci (M^3)$. 

\item By a partition of unity argument, we can assume that $g$ is compactly supported in $O$ where $O$ is the open set defined after Lemma \ref{lem:etap0}. 

\item  Replacing $g$ by $g \delta$ we have that 
\begin{gather} \label{eq:intacalculer}
\int_{O_{x,z}} T_k ( x, \cdot ,z)  \mu  = \int_{O_{x,z}} F^k ( x,\cdot ,z) g(x, \cdot ,z) \; \mu_{x,z} 
\end{gather}
where $\mu_{x,z}$ is the measure introduced above.
\item By Lemma \ref{lem:indep_E} and Equation (\ref{eq:EEE}), we can assume that 
\begin{gather} \label{eq:Fsimplifie}
  F =  e ^{ -\frac{1}{4}  ( w +  v ) \cdot ( \con{w} -  \con{ v} ) } \pi^* E
\end{gather}
\item Doing the Taylor expansion of $y \rightarrow g ( x, y, z)$ in terms of the coordinates $w_j $, $\con{w}_j $, we obtain that 
$$ g (x,y,z) = \sum _{|\al | + | \be | \leqslant 2N }  g_{\al,\be} ( x,z)   w ^ \al (x,y,z)  \con{w}^\be (x,y,z) +  r_N (x,y,z) $$
where $r_N =  \bigo(|w|^{2N+1})$. Choose $\rho \in \Ci_0 ([0,\epsilon))$ equal to 1 on a sufficiently large neighborhood of $0$ so that $g \rho ( |w|) = g $. Then $\rho ( |w|) r_N$ is compactly supported in $O$. Since $\Si \cap O = \{ v = w = 0 \}$, $r_N$ vanishes to order $2N+1$ along $\Si$ and by Proposition \ref{prop:estim_symb}, $F^k \rho ( |w|) r_N$ belongs to $\bigoinf ( k^{-N - 1/2})$. So by Lemma \ref{lem:est1}, we may assume that $g$ has the form
\begin{gather} \label{eq:gsimplifie}
 g (x,y, z) =  \rho ( |w(x,y,z)|)   g_{\al,\be} ( x,z)   w ^ \al (x,y,z)  \con{w}^\be (x,y,z) 
\end{gather}
for some $\al$, $\be$ and a smooth function $g_{\al, \be}$.

\end{itemize}
Now inserting (\ref{eq:Fsimplifie}) and (\ref{eq:gsimplifie}) in (\ref{eq:intacalculer}), we obtain that 
$$  \int_{O_{x,z}} T_k ( x, \cdot ,z)  \mu  =  E^k ( x,z) g_{\al, \be} ( x,z) I_{\al, \be} (k, v(x,z) ) $$
with 
$$  I_{\al, \be} (k,v ) = \int_{B_{\epsilon}}  e ^{ -\frac{k}{4}  ( w +  v ) \cdot ( \con{w} -  \con{ v} ) } \rho ( |w|) w^\al \con{w}^\be \; \mu_n (w)  $$
where $\mu_n$ is the measure $| dz^1 d\con{z}^{1} \ldots dz^n d\con{z}^n|$ of $\C^n$

\begin{lemme} \label{lem:calcul_integrale}
There exists $C>0$ such that 
\begin{gather} \label{eq:but}
 I_{\al, \be} (k, v ) = \Bigl( \frac{2 \pi}{k} \Bigr)^n \sum_{\ell  = 0 }^{ m } k^{-\ell}  P_{\ell} ( -v,  \con{v}) + e^{ \frac{k}{4} |v|^2} \bigo ( e^{- k /C } )
\end{gather}
where $m = \min ( |\al|, |\be|)$, $P_{\ell} (w, \con{w} ) =  \frac{4^\ell}{\ell !} \Delta ^{\ell }( w^\al \con{w}^\be) $
\end{lemme}

\begin{proof} 
We have 
$$ I_{\al, \be}  (k,v) = \int_{\C^n }  e ^{ -\frac{k}{4}  ( w +  v ) \cdot ( \con{w} -  \con{ v} ) }
 w^\al \con{w}^\be \;  \mu_n (w)  + R_k ( v) $$
We compute explicitly the integral on $\C^n$ by applying Lemma \ref{lem:int}, which gives the sum in the right hand side of (\ref{eq:but}).  The remainder is 
$$ R_k ( v) = \int_{\C^n } (1- \rho (|w|)) e ^{ -\frac{k}{4}  ( w +  v )  \cdot ( \con{w} -  \con{ v} ) }
 w^\al \con{w}^\be \;  \mu_n (w) $$
where we have extended $\rho$ to $\R$ by setting $\rho (t) =0$ for $t \geqslant \ep$. 
One has $| e^{ -  ( w +  v ) \cdot ( \con{w} -  \con{ v} ) } | = e^{ - |w|^2 + |v|^2}$ so 
$$ |R_k ( v ) | \leqslant e^{ \frac{k}{4} |v|^2} \int_{\C^n } e^{- \frac{k}{4} |w|^2} (1- \rho (|w|)) |w^\al \con{w}^\be|  \mu_n (w)
$$
Let $\delta>0$ such that $\rho ( t) =1 $ for $ t \leqslant \delta$. Then 
\begin{xalignat*}{2}
 |R_k ( v ) | & \leqslant  e^{ \frac{k}{4} |v|^2} e^{-\frac{k}{8} \delta^2} \int_{\C^n } e^{- \frac{k}{8} |w|^2} (1- \rho (|w|)) |w^\al \con{w}^\be|  \mu_n (w) \\
 & \leqslant  e^{ \frac{k}{4} |v|^2} e^{-\frac{k}{8} \delta^2} \int_{\C^n } e^{- \frac{k}{8} |w|^2}  |w^\al \con{w}^\be|  \mu_n (w) \\ 
& =  A e^{ \frac{k}{4} |v|^2} e^{-\frac{k}{8} \delta^2} k^{ - \frac{1}{2} ( |\al | + | \be |) -n } \text{ with } A =  \int_{\C^n } e^{- \frac{1}{8} |w|^2}  |w^\al \con{w}^\be|  \mu_n (w)
\end{xalignat*} 
which concludes the proof with $C$ any real number greater than $8/\delta^2$. 
\end{proof}

To conclude, if $T_k = F^k g$ with $F$ given by (\ref{eq:Fsimplifie})  and $g$ by (\ref{eq:gsimplifie}), then  
$$ k^n U_k  =  (2 \pi )^n E^k  g_{\al, \be}  \sum_{\ell}^m k^{-\ell} P_{\ell} (-u, \con{u}) +  \bigo ( e^{-k/C})  $$
where each $P_{\ell}$ is a polynomial of order $|\al|+ |\be| -2\ell$. Here we have used that $|E^k| = e^{-k|u|^2/2}$ by Lemma \ref{lem:etap-1}. This ends the proof of implication (\ref{eq:second_equation}) and the proof of Theorem \ref{theo:main1new}.

\subsubsection*{Summary}

To prepare the symbol computation, let us summarize the proof of Theorem \ref{theo:main1new}. Consider $R \in \Fa_m$ and $S \in \Fa_p$. Their Schwartz kernel have the form 
$$  \Bigl( \frac{k}{2 \pi} \Bigr)^n \Bigl( E^k r (x, k^{-1}, u, \con u ) + \bigo ( k^{ - m- \frac{1}{2}}) \Bigr) , \quad \Bigl( \frac{k}{2 \pi} \Bigr)^n \Bigl( E^k  s (x,k^{-1},   u,  \con u ) + \bigo ( k^{ - p-\frac{1}{2}}) \Bigr) $$
where $r$ and $s$  are polynomials in the variable $k^{-1}$, $u_i$, $\con{u}_i$ obtained by linearizing the coefficients $f_\ell$ in (\ref{eq:noyauFdev}). 
Let $T_k$ be defined by Equation (\ref{eq:def_T}). 
Then it follows from Equation (\ref{eq:rel_u_v_w}) that
$$ T_k  = \Bigl( \frac{ k }{ 2 \pi} \Bigr)^{2n} \Bigl( F^k t( x, k^{-1} , v,\con {v} , w, \con{w}) + \bigo( k ^{-m-p-\frac{1}{2}}) \Bigl) $$ 
with 
$$ t( x, \hb,  v,\con {v} , w, \con{w}) =   r \bigl(  x, \hb , \frac{ w + v}{2}  ,  \frac{ \con{w} + \con{v} }{2}  \bigr) s \bigl( x, \hb, \frac{- w + v}{2}  , \frac{-\con w + \con{v} }{2}  \bigr). $$
Since the function $\delta$ is equal to $1$ on $\Si$, we also have
$$  T_k  = \Bigl( \frac{ k }{ 2 \pi} \Bigr)^{2n} \Bigl( F^k \frac{t( x, k^{-1} , v,\con {v} , w, \con{w})}{\delta (x)} + \bigo( k ^{-m-p- \frac{1}{2}}) \Bigl) $$ 
The Schwartz kernel $U_k$ of $R_k S_k$ satisfies Equation (\ref{eq:det_U}). 
By Lemma \ref{lem:calcul_integrale}, 
\begin{gather} \label{eq:conclusion} 
 U_k  = 
\Bigl( \frac{ k }{ 2 \pi} \Bigr)^{n} \Bigl( E^k  \sum_{\ell = 0 } ^{ m+p } \frac{4^{\ell}}{k^{\ell}\ell ! }  (\Delta^\ell t) (x,k^{-1} ,  u, \con {u} , -u , \con {u} )+ \bigo( k ^{-m-p- \frac{1}{2}}) \Bigl) 
\end{gather}
Here $\Delta = \sum \partial_{w^i} \con{\partial} _{w^i}$ acts on the variable $w$ and $\con{w}$ of $t ( x, \hb, v, \con v , w, \con {w})$.

\subsection{Symbolic calculus} 

Let us define the symbol of the elements of $\Fa$ as in Section \ref{sec:four-herm-familly}. We will identify the normal bundle to the diagonal of $M^2$ with the tangent bundle of $M$ by the isomorphism $TM \rightarrow N( \op{diag} M)$ sending $X$ into $[X,0]$. So we have a symbol map $\si_m $ from $\Fa_m$ to 
$$ \Wi_m ( M,  A) =  \bigoplus_{ \ell \in \Z \cap [-m, m/2]} \hb ^\ell \Ci ( M , S^{m-2\ell} ( T^* M ) \otimes \op{End} A) 
$$ 
To describe this more explicitly, introduce a local frame $\partial_1, \ldots, \partial_n$ of $T^{1,0} M$. Let $(z_i)$ be the dual frame of $(T^{1,0} M)^*$. Then if the Schwartz kernel of $P_k$ is given by Equation (\ref{eq:noyauFdev}), we have 
$$ \si_m ( P_k)  = \sum _{ \ell \in \Z \cap [-m, m/2]} \hb ^\ell [f_\ell] $$
with 
\begin{gather} \label{eq:symb}
 [f_\ell]  ( z, \con{z}) = \sum_{|\al| + |\be| = m - 2 \ell } \frac{1}{\al ! \be !} \bigl( (\partial^\al  \con{\partial}^\be \boxtimes \op{id} )
 f_\ell \bigr) |_{\op{diag} M} z^ \al \con{z}^\be 
\end{gather}
and  $ \partial^\al = \partial_1^{\al(1)} \ldots  \partial_n^{\al( n) }$, $\con{\partial}^\be = \con{\partial}_1 ^{\be(1)} \ldots  \con{\partial}_n^{\be ( n)}$. By lemma \ref{lem:indep_E} the symbol map is defined independently of the choice of $E$.

Applying fibrewise the product $\star$ defined in Equation (\ref{eq:prodfor}), we obtain bilinear applications
$$ \Wi_m ( M,  A) \times \Wi_p ( M , A) \rightarrow \Wi_{m + p } ( M ,  A) $$
that we still denote by $\star$. Let us define also the adjoint map
$$ \Wi _m  (M ,  A) \rightarrow \Wi_m ( M ,  A ) , \qquad  \hb^\ell z^\al \con{ z}^ \be a \rightarrow \hb^{\ell} z^{\be} \con{z} ^\al a^*$$
and denote it by $*$.  

\begin{theo} \label{theo:main1new2}
For any $(R_k) \in \Fa_{m}$ and $( S_k ) \in \Fa_{p}$, we have 
$$ \si_m ( R_k ^*) = \si_m ( R_k) ^*, \qquad \si_{m+p } ( R_k S_k) = \si _m ( R_k ) \star \si_p ( S_k).$$
\end{theo}

\begin{proof} 
This follows directly from Equation (\ref{eq:conclusion}). Indeed, by equation (\ref{eq:coordonnee_u_1}), the function $r$ and $s$ appearing in Equation (\ref{eq:conclusion}) are the symbols of $\si_m(R_k)$ and $\si_p(S_k)$ respectively. By the same reason, the symbol $\si= \si_{m+p}(R_k S_k)$ is given by  
$$ \sigma ( \hb, u , \con{u} ) =    \sum_{\ell = 0 } ^{ m+p} \frac{4^\ell \hb^{\ell}}{\ell !}  (\Delta^\ell t) (x, \hb, u , \con {u} , -u , \con u  )$$
 To recover Formula (\ref{eq:prodfor}), instead of the variable $v$, $\con {v}$, $w$, $\con w$, we use 
$ z = v$, $\con{z} = \con{v}$, $\zeta = - (w+v)/2$, $\con \zeta = (-\con {w} + \con{v}) /2 $.
\end{proof}

\subsection{Norm estimates} 

Let $\mathcal{L}^2 ( M, A_k)$ be the space of $L^2$-sections of $A_k$. Endow $\mathcal{L}^2 (M, A_k)$ with the scalar product defined by integrating the pointwise scalar product against the Liouville measure $\mu$. Any bounded operator of $\mathcal{L}^2 ( M , A_k)$ has a corresponding norm that we call the operator norm.
For any $(P_k) \in \Fa_{0} (M, L,j ,  A)$ and $k \in \N$, since $M$ is compact and $P_k$ has a smooth kernel, $P_k$ extends to a bounded operator of $\mathcal{L}^2 ( M, A_k)$. We denote again this operator by $P_k$ hoping it will not create any confusion.

\begin{prop} \label{prop:norme_estimation}
For any $m \in \N$ and $(P_k) \in \Fa_{m} (M, L,j ,  A)$,
the operator norm of $(P_k)$ is in $\bigo ( k ^{-m/2})$. 
\end{prop}

\begin{proof}
Recall the Schur test. Let $P : \Ci ( M , F) \rightarrow \Ci ( M, F)$ be an operator with Schwartz kernel $K$ as in (\ref{eq:noyauPK}). Assume that there exists $C$ such that for any $x \in M$ one has
$$ \int_M |K(x , \cdot ) | \mu \leqslant C, \qquad \int_M |K( \cdot ,x) | \mu \leqslant C. $$  
Then the operator norm of $P$ is bounded above by $C$. 

Let $(P_k) \in \Fa_{m} (M, L,j ,  A)$. Assume that its Schwartz kernel is $k^{n - \ell } E^k f $ with $\ell \in \Z$ and $f \in \Ci ( M^2, A \boxtimes \con{A})$ vanishing to order $d \geqslant m - 2 \ell$ along the diagonal. Choose a coordinate system $( x_i)$ on an open set $U$ of $M$ and assume that $f$ has a compact support contained in $U^2$. Write $\mu = \rho |dx_1 \ldots dx_{2n}|$. Then there exists $C$ such that
\begin{gather}\label{eq:maj_symb} 
 | f(x,y) \rho (y)  | \leqslant C | x - y | ^d, \qquad \forall (x,y) \in U^2 
\end{gather}
and 
$$ |E(x,y)| \leqslant C e^{-|x-y|^2/C} \qquad \forall (x,y) \in \supp f.$$
Consequently, for any $x\in U$,  
\begin{xalignat*}{2}
 \int_M |P_k( x, \cdot )| \mu & \leqslant C^2 k ^{n-\ell} \int_M e^{-k |x-y|^2/C } |x-y|^d |dx_i| \\
 & \leqslant C^2 k ^{n-\ell} \int_{\R^{2n}} e^{-k |x-y|^2/C } |x-y|^d |dx_i| \\
 & \leqslant C^2 k ^{- \ell - d/2} \int_{\R^{2n}} e^{- |x-y|^2/C } |x-y|^d |dx_i| \\
   & \leqslant C' k^{-m/2}
\end{xalignat*}
Similarly, $\int_M |P_k( x, \cdot )| \mu  \leqslant C'' k^{-m/2}$ which shows that $\| P_k \| = \bigo ( k^{-m/2})$. It is easy now easy to deduce the same result for any $(P_k) \in  \Fa_{m} (M, L,j ,  A)$ by using the expansion (\ref{eq:noyauFdev}) and a partition of unity. 
\end{proof}

\subsection{Commutator identity} 
Consider now a function $f \in \Ci (M)$. Let $X$ be its Hamiltonian vector field. Introduce  a derivative $D_X : \Ci ( M , A) \rightarrow \Ci ( M, A)$ in the direction of $X$. Assume that $D_X$ preserves the metric of $A$, meaning that
$$ X. (\si , \tau) = ( D_X \si, \tau ) +  ( \si , D_X \tau ) , \qquad \forall  \si , \tau \in \Ci ( M , A) .$$
Denote by $P_{X,k}$ the derivation 
$$ P_{X,k} = ( \nabla^{L^k}_X \otimes \op{id} + \op{id} \otimes D_X) : \Ci ( M , L^k \otimes A) \rightarrow \Ci ( M , L^k \otimes A ) . $$

\begin{lemme} \label{lem:lemquifaitmal}
For any $Q \in \Fa_0$, the family $( \bigl[ f + \frac{i}{k} P_{X,k} , Q_k \bigr] )$ belongs to $\Fa_2 $. Using the same notations as in Equation (\ref{eq:symb}), its symbol is
$$ \Bigl( \sum \om ( \con{\partial}_i, [ \con{\partial}_j,X] ) \con{z}_i \con{z}_j - \om( \partial_i, [\partial_j, X]) z_i z_j \Bigr) \si_0  ( Q) + i \hb  D_X  \si_0(Q)   .$$
\end{lemme} 

\begin{proof} 
The Schwartz kernel of $ \frac{i}{k} P_{X,k} Q_k $ is $( \frac{i}{k} P_{X,k} \boxtimes \op{id} )Q_k$. Using that $\nabla$ preserves the metric of $L$ and that $D$ preserves the metric of $A$, we prove that the Schwartz kernel of $\frac{i}{k} Q_k  P_{X,k}$ is 
$$  -\tfrac{i}{k} ( \op{id} \boxtimes \con{P}_{X,k}  ) Q_k  -  \tfrac{i}{k} ( 1 \boxtimes \op{div}(X) ) Q_k  ,$$
where $\op{div} (X)$ is the divergence of $X$ with respect to the Liouville measure. Since $X$ is a symplectic vector vector field, $\op{div} X = 0$. 

Consequently, the Schwartz kernel of $ \bigl[ f + \frac{i}{k} P_{X,k} , Q_k \bigr]$ is 
$$ \bigl[ ( f \boxtimes 1 - 1 \boxtimes f)  + \tfrac{i}{k} ( P_{X,k} \boxtimes \op{id} + \op{id} \boxtimes \con{P}_{X,k} )  \bigr] Q_k $$
So the result is a consequence of Lemma \ref{lem:derhamE}. Indeed, observe that the Hamiltonian vector field of $f \boxtimes 1 - 1 \boxtimes f$ for the symplectic structure of $M \times M^{-}$ is the vector field $(X,X)$. The computation of the symbol follows from Proposition \ref{prop:secE_com}.  
\end{proof}

\section{Existence of the projector}  \label{sec:existence-projector}

In this section we prove the following result. 

\begin{theo} \label{theo:existence-projector} 
For any  compact symplectic manifold $M$ with a prequantum bundle $L \rightarrow M$, a Hermitian vector bundle $A \rightarrow M$ and an almost complex structure $j$,
there exists an operator $( \Pi_k ) \in \Fa_0 (M,L,j , A)$ with symbol $1_A$ and satisfying for any $k$, $\Pi_k^ 2 = \Pi_k $ and $\Pi_k ^* = \Pi_k$. 
\end{theo} 

Recall that $1_A$ is the section of $\op{End} A$ such that $1_A (x) $ is the identity of $A_x$ for any $x$. Consider an operator $(\Pi_k)$ satisfying the condition in Theorem \ref{theo:existence-projector}. Since $\Pi_k$ has a continuous kernel and $M$ is compact, $\Pi_k$ is trace class. So the image $\mathcal{H}_k $ of $\Pi_k$ is finite dimensional. Hence Theorem \ref{theo:existence_proj} is a consequence of Theorem \ref{theo:existence-projector}. 
Furthermore, 
$$ \op{dim} ( \mathcal{H}_k) = \op{trace} (  \Pi_k  ) = \int_M g_k ( x) \mu (x) $$
 where $g_k (x) $ is the trace of $ \Pi_k   ( x,x) \in \op{End} A_{k,x}$ and $\mu=\om^n/ n !$ is the Liouville measure.
 So we deduce that 
$$ \op{dim} ( \mathcal{H}_k ) = \Bigl ( \frac{k}{2 \pi} \Bigr)^n  \bigl( \op{ rank} (A) \bigr)\int_M \mu + \bigo ( k^{-n-1}).$$
 where $2n$ is the dimension of $M$.

\subsection{Proof of Theorem \ref{sec:existence-projector}, step 1} 

Consider the usual data $(M, L , j , A)$ and let $\Fa_0 = \Fa_0 (M,L,j,  A)$ be the associated algebra.
Choose $P =(P_k)  \in \Fa_0$ with symbol $1_A$. Replacing $P_k$ by $\tfrac{1}{2} ( P_k + P_k^* )$ if necessary, we may assume that $P_k$ is formally self-adjoint. 
Recall that $P_k$ extends to a continuous operator of $\mathcal{L}^2( M , A_k) $. It is self-adjoint and by Proposition \ref{prop:norme_estimation}, its norm is bounded independently of $k$.

\begin{prop} \label{prop:spectre_almost_proj}
There exists $C>0$ such that for any $k$, the spectrum of $P_k$ is contained in $[-C k^{-1/2}, C k^{-1/2}] \cup [1 - C k^{-1/2}, 1 + C k^{-1/2} ]$. 
\end{prop} 

\begin{proof}  Since the symbol of $P$ is $1_A$, by Theorem \ref{theo:main1new2}, $P^2 - P$ belongs to $\Fa_1$. So by Proposition \ref{prop:norme_estimation}, the operator norm of $P_k^2 - P_k$ is a $\bigo ( k^{-1/2})$.  Recall that for any bounded operator $R$ and polynomial $P$ 
$$  \| P (R ) \|  = \sup _{ \lambda \in \sigma (R)} | P ( \lambda) | ,$$
cf. as instance Lemma 2, page 223 of \cite{ReSi}. Applying this to $R = \Pi_k$ and $P( X) = X^2 - X$, we get the result. 
\end{proof}  

Replacing $P_k$ by $0$ for a finite number of values of $k$, we may assume that for any $k$, the spectrum of $P_k$ is contained in $I_0 \cup I_1$ with $I_0$ and $I_1$ the following neighborhoods of $0$ and $1$ respectively: $I_0  = [-1/4, 1/4]$, $I_1=  [3/4, 5/4]$.  
  
In the sequel, we will use the functional calculus for operators of a Hilbert space. Since we consider only bounded operators, this is relatively easy, cf. for instance Theorem VII.1 of \cite{ReSi}. Recall that if $f$ and $g$ are two continuous functions from $\R$ to $\C$ and $Q$ is a bounded self-adjoint operator of a Hilbert space $H$, then 
$$ f ( Q ) . g ( Q ) = ( fg ) ( Q ) .$$ 
Recall also that if $ f = g$ on the spectrum of $Q$, then $ f( Q ) = g ( Q)$.

Let $\chi : \R \rightarrow \R$ be any continuous function which is equal to $0$ on $I_0$ and equal to $1$ on $I_1$. We will define our projector as $$\Pi_k = \chi ( P_k).$$ Since $\chi $ is real, $\chi ( P_k)$ is a bounded self-adjoint operator. Since $\chi^2 = \chi$ on $I_0 \cup I_1$, 
$$   \chi ( P_k ) ^2 = \chi ( P_k) .$$
It remains to prove that $\chi (P) \in \Fa_0$.

\subsection{Proof of Theorem \ref{theo:existence-projector}, step 2} 

The idea is to express $\chi ( P_k ) - P_k$ in terms of $Q_k = P_k^2 - P_k$. 
Consider the function 
$$ f ( y) = \frac{1}{2} \bigl( 1 - ( 1 + 4 y )^{-1/2} \bigr), \qquad y > -1 /4 .$$
Observe that for any $x \in \R \setminus \{ 1/2 \}$, $y = x^2 - x  > -1/4$. A direct computation shows that 
$$ x +  (1-2x) f( x^2 -x )  = \begin{cases} 1 \text{ if } x > 1/2 \\ 0 \text{ if } x < 1/2  \end{cases} $$
so that on $I_0 \cup I_1$,  we have
$ \chi (x) =  x +  (1-2x) f( x^2 -x ) $. Consequently, 
\begin{gather} \label{eq:the_equation}
 \chi ( P_k) = P_k +  ( \op{id} - 2 P_k) f ( Q_k )  .
\end{gather}
We learned this formula  in  \cite{Bo2}.  
For any positive integer $m$, let us write
$$ f( y ) = p_m ( y) + y^{m+1} f_m ( y)$$
where $p_m$ is a polynomial with degree $m$ and $f_m$ a continuous function on $(-1/4, \infty )$.
Then we have 
$$ \chi ( P_k) = \chi_m (P_k) + r_m ( P_k) \quad \text{ where } \quad \begin{cases} \chi_m (P_k) = P_k + (\op{id} - 2 P_k)  p_m ( Q_k) \\ r_m ( P_k) = ( \op{id} - 2 P_k) Q_k^{m+1} f_m ( Q_k)
\end{cases}
$$

\begin{lemme} \label{lem:p1}
For any positive integers $m < m'$, we have 
$$ \chi_m ( P ) = \chi_{m'} (P) \mod \Fa_{m+1} .$$
\end{lemme} 

\begin{proof} 
This follows from $p_m (X) = p _{m'} (X) \mod X^{m+1}$, Theorem \ref{theo:main1new2} and the fact that $Q \in \Fa_1$. 
\end{proof}

\begin{lemme} \label{lem:p2}
For any $m \geqslant 1 $ and $k \in \N$, $r_m ( P_k)$ has a smooth Schwartz kernel. Furthermore, for any $m \geqslant 1$, the Schwartz kernel family $( r_m ( P_k) , k \in \N)$ is in $\bigoinf ( k^{ 2n - ( m+1)/2})$. 
\end{lemme} 

Lemmas \ref{lem:p1} and \ref{lem:p2} imply that $\chi (P )$ belongs to $\Fa_0$. This concludes the proof of Theorem \ref{theo:existence-projector}. 

\subsection{Proof of Lemma \ref{lem:p2}} 

Consider any Hermitian rank $r$ vector bundle $F \rightarrow M$. Denote by $\Di ( M, F)$ the space of $F$-valued distributions. Let $\pi : \Ci ( M, \con F ) \times \Di (M, {F}) \rightarrow \C$ be the continuous bilinear map given for smooth sections $f$, $g$ by 
$$ \pi ( f, g) = \int_M f(x) \cdot g(x) \mu(x) .$$
where the dot stands for the contraction $\con{F}_x \otimes F_x \rightarrow \C$ induced by the metric. 

By the Schwartz kernel theorem (section 5.2 and Theorem 5.2.6 of \cite{Ho}), there is a one to one correspondence between the space of continuous operators $\Di ( M , F) \rightarrow \Ci ( M , F)$ and  $\Ci ( M^2, F \boxtimes \con{F})$. If $A$ and $K_A$ are the corresponding operator and kernel, we have for any $ f \in \Di ( M , F)$, 
\begin{gather} \label{eq:sk1} 
  A( f ) ( x) = \pi ( K_A ( x , \cdot) , f) 
\end{gather}
Furthermore 
\begin{gather} \label{eq:sk2} 
 K_A ( \cdot, y ) = A ( \delta_y) 
\end{gather} 
where $\delta _y \in \con{F}_y \otimes \Di ( M , F)$ is such that $ \pi ( \delta_y , f ) = f(y)$ for any $f \in \Ci ( M , \con F )$. 

Let $\mathcal{L}^2 ( M, F)$ be the space of $L^2$-sections of $F$. Endow $\mathcal{L}^2 (M, F)$ with the scalar product obtained by integrating the pointwise scalar product against the Liouville measure $\mu$. 
As an  application of Equations (\ref{eq:sk1}) and (\ref{eq:sk2}), we have the following lemma. 

\begin{lemme} \label{lem:schw-kern}
Let $A$ and $C$ be two continuous operators $\Di (M , F) \rightarrow \Ci ( M , F)$ with Schwartz kernels $K_A$ and $K_C$ respectively. Let $B$ be a bounded operator of $\mathcal{L}^2 ( M , F)$. Then $ABC$ is continuous $\Di (M , F) \rightarrow \Ci ( M , F)$. Its Schwartz kernel is given by 
$$ K_{ABC} = \pi ( K_A ( x, \cdot) , B ( K_C( \cdot , y ))) .$$
\end{lemme} 
 It implies in particular that
\begin{gather} \label{eq:estim_schwartz_kernel}
  \| K_{ABC} \|_{\infty} \leqslant \op{vol} (M) \| K_A \|_\infty \| K_C \|_\infty \| B \| . \end{gather}
where for any kernel $K$, $\| K \|_{\infty} = \sup _{z \in M^2}  |K(z)|$, and $\|B \| $ is the operator norm of $B$.    
Furthermore if $D$ and $D'$ are any differential operators of $\Ci ( M , \con {F})$ and $\Ci ( M , F)$ respectively, we have
\begin{gather*} 
  ( (D \boxtimes D'). K_{ABC} )(x,y)  =  \pi \bigl(  (D.K_A)( x , \cdot) , B ( (D'.K_C)( \cdot, y)) \bigr) \end{gather*}
so that we have
\begin{gather} \label{eq:deriv_schwartz_kernel}
 \| (D \boxtimes D'). K_{ABC} \|_{\infty} \leqslant \op{vol} (M) \| D. K_A \|_\infty \|D'. K_C \|_\infty \| B \| .
\end{gather}

Consider now families of operators. Introduce an Hermitian line bundle $L \rightarrow M$ and set $F_k = L^k \otimes F$. Consider three families $A$, $B$ and $C$ of continuous operators  
$$  A_k , C_k  : \Di ( M , F_k) \rightarrow \Ci ( M , F_k), \quad B_k : \mathcal{L}^2 ( M, F_k) \rightarrow \mathcal{L}^2 ( M , F_k), \quad k \in \N. $$
By Lemma \ref{lem:schw-kern}, for any $k$, $A_k B_k C_k $ has a smooth Schwartz kernel $K_k$. Assume that $\| B_k \| = \bigo (1)$. Then by Equation (\ref{eq:estim_schwartz_kernel}), if the Schwartz kernel families of $A$ and $C$ are in $\bigo ( k^{N})$ and $\bigo ( k^{N'})$ respectively, the family $(K_k)$ is in $\bigo ( k^{N + N'})$. Similarly by Equation (\ref{eq:deriv_schwartz_kernel}), if the Schwartz kernel families of $A$ and $C$ are in $\bigoinf ( k^{N})$ and $\bigoinf ( k^{N'})$ respectively, the family $(K_k)$ is in $\bigoinf ( k^{N + N'})$.

\begin{proof}[Proof of Lemma \ref{lem:p2}]
Write $$r_m ( Q_k ) = Q_k ^{m} f_m ( Q_k) Q_k - 2 P_k Q_k^m f_m (Q_k) Q_k .$$ Since $Q$ is in $\Fa_1$, $Q^m$ is in $A_m$. So the Schwartz kernel family of $Q^m$ is a $\bigo ( k^{n - m /2})$. Similarly, the Schwartz kernel families of $P Q^m$ and $Q$ are respectively in $\bigo ( k^{n - m/2})$ and $\bigo ( k^{n- 1/2})$. Furthermore, $\| f_m(Q_k)\| = \bigo ( 1)$. Applying the previous considerations to $A = Q^m$ or $PQ^m$, $B = f_m ( Q)$ and $C = Q$, we get the conclusion. 
\end{proof}

\section{Toeplitz Operators} \label{sec:toeplitz-operators}

Consider a  compact symplectic manifold $M$, a prequantum bundle $L \rightarrow M$, a Hermitian vector bundle $A \rightarrow M$ and an almost complex structure $j$. Introduce the associated algebra $\Fa_0 =   \Fa_0 ( M, j, L , A)$. Let $\Pi \in \Fa_0 $ be self-adjoint, idempotent and with symbol $1_A$. 

A Toeplitz operator associated with these data is by definition an operator $T \in \Fa_0$ such that $\Pi T \Pi = T$. We denote by $\To = \To ( M , j , L , A , \Pi)$ the corresponding set of Toeplitz operators. Since by Theorem \ref{theo:main1new}, $\Fa_0 $ is closed by composition and $\Pi^2 = \Pi$, $\To$ is an algebra with unit $\Pi$.

\subsection{Covariant and contravariant symbols} 

The following useful lemma is an immediate consequence of Lemma \ref{lem:ex_proj_symb}. 
 
\begin{lemme} \label{lem:usefull}
For any $m \in \N$, for any $T \in \Fa_m \cap \To$, we have the following
\begin{itemize} 
\item if $m $ is odd, then $\si_m (T) = 0$ so that $T \in \Fa_{m+1}$. 
\item if $m$ is even, $ \si _m (T) = \hb^{m/2} f$ for some $f \in \Ci ( M , \op{End} A )$. 
\end{itemize} 
\end{lemme}

We define now the contravariant and covariant symbol of a Toeplitz operator. Let us denote by $\Symb ( M, \op{End } A  )$ the set of sequences $(f( \cdot, k))$ of  $\Ci (M, \op{End} A )$ admitting an asymptotic expansion for the $\Ci$ topology of the form 
$$ f( \cdot, k ) = f_0 + k^{-1} f_1 + k^{-2} f_2 + \ldots, \qquad \text{ with }f_\ell \in \Ci ( M, \op{End} A ), \;\forall \ell \in \N .$$ 
By Borel Lemma, the map from $\Symb ( M, \op{End } A  )$ to $\Ci ( M, \op{End } A  ) [[\hb]]$ sending $(f(\cdot, k))$ to the formal series $\sum \hb^\ell f_\ell$ is onto.

\begin{prop} \label{prop:contravariant}
For any sequence $(f( \cdot, k) )$ in $\Symb (M , \op{End } A )$, 
the operator 
$$ T= ( \Pi_k M_{f( \cdot, k)} \Pi_k: \Ci ( M , L^k , A) \rightarrow \Ci ( M , L^k \otimes A)  )$$ is a Toeplitz operator. Furthermore, the coefficients $f_\ell$, $\ell \in \N$ of the asymptotic expansion of $f( \cdot, k)$ are uniquely determined by $T$. More precisely, for any $m$, $T \in \Fa_{2m}$ if and only if $f_\ell = 0$ for any $\ell < m$. If it is the case, $\si_{2m}(T) = \hb^{m} f_m$.  

Conversely, any Toeplitz operator $T \in \To$ has the form $( \Pi_k M_{f( \cdot, k)} \Pi_k)$ up to $\Fa_\infty \cap \To$ for some $(f (\cdot, k ) ) \in \Symb (M , \op{End} A )$.
\end{prop} 

In this statement, we denoted by $\Fa_{\infty}$ the space 
$ \Fa_{\infty} = \bigcap_{m \in \N}  \Fa_m$ which consist in the families (\ref{eq:fam_op}) with a Schwartz kernel in $\bigoinf (k^n) \cap \bigo ( k^{-\infty}) = \bigoinf ( k^{-\infty})$. 

This defines  an application
$$ \si_{\op{cont}} : \To \rightarrow \Ci ( M  , \op{End} A ) [[\hb ]] , \quad  (\Pi M_{f( \cdot, k )} \Pi_k) \rightarrow \sum \hb^\ell f_\ell$$ 
called the contravariant symbol map. It is onto and its kernel is  $\Fa_\infty \cap \To$.

\begin{proof} The Schwartz kernel of $M_f \Pi_k$ is the section $ ( x, y) \rightarrow f(x) . \Pi_k(x,y)$.  So the operator $(M_{f( \cdot, k) } \Pi_k)$ belongs to $\Fa_0$. By  Theorem \ref{theo:main1new}, we conclude  that $T= (\Pi_k M _{f( \cdot, k)} \Pi_k) \in \To$. 

Assume that the coefficients of the asymptotic expansion of $f( \cdot , k)$ satisfy  $f_\ell = 0$ for any $\ell < m$. Then considering again the Schwartz kernel of $M_{f( \cdot, k)} \Pi_k$, one sees that $(M_{f( \cdot, k )} \Pi_k ) \in \Fa_{2m} $ and  $\si_{2m} (M_{f( \cdot, k)} \Pi_k ) = \hb  ^m f_m$. Applying again Theorem \ref{theo:main1new}, it follows that  $T \in \Fa_{2m}$ and by Theorem \ref{theo:main1new2} $\si_{2m} (T) = \hb  ^m f_m$. Furthermore, by Lemma \ref{lem:usefull},  $\si_{2m} (T) = 0 $ if and only if  $T \in \Fa_{2 (m+1)}$. 

Arguing by induction on $m$, we deduce that $T \in \Fa_{2m}$ if and only if $f_0 = \ldots = f_{m-1} =0$. In that case, $\si_{2m} (T) = \hb^m f_m$. 

If $T$ is any Toeplitz operator in $\Fa _p$, then by lemma \ref{lem:usefull}, we can assume that $p =2m$ and we have $\si _{2m} (P) = \hb^m f_m$ for some function $f_m \in \Ci(M)$. So $T = (\Pi_k M_{k^{-m} f_m} \Pi_k)$  up to $\Fa_{2(m+1)} \cap \To$. Arguing by induction on $m$ and using Borel Lemma, we conclude that $T$ has the form $( \Pi_k M_{f( \cdot, k)} \Pi_k)$ up to $\Fa_\infty \cap \To$.
\end{proof}

\begin{prop} \label{prop:covariant}
For any Toeplitz operator $ T \in \To$ with Schwartz kernel $(K_k \in \Ci ( M^2, ( L \boxtimes \con{L} )^k \otimes ( A \boxtimes \con{A})) , \; k \in \N)$, the restriction of $( \bigl( \frac{2 \pi } {k} \bigr)^n K_k) $ to the diagonal belongs to $  \Symb ( M , \op{End} A)$. The corresponding application 
$$ \si : \To \rightarrow \Ci ( M , \op{End} A) [[\hb]], \quad T \rightarrow \sum \hb^{\ell} f_\ell $$
is onto and its kernel is $\To \cap \Fa_{\infty}$. Furthermore for any $m$, $T \in \Fa_{2m }$ if and only if $\si  (T) \in \bigo ( \hb^{m})$. And if it is the case, $\si (T)  = \si_{2m }(T) + \bigo (\hb^{m+1}) $.
\end{prop} 

$\Ci(M , \op{End} A) [[\hb]]$ has an associative product induces by the pointwise product of $\Ci(M, \op{End} A)$ is an algebra. Since $\si ( \Pi) = 1_A + \bigo ( \hb)$, $\si (\Pi)$ is invertible. Define the covariant symbol map by
$$ \si_{\op{cov}}  : \To \rightarrow \Ci ( M , \op{End} A) [[\hb]], \qquad \si_{\op{cov}} ( T) = \si ( T) . \si (\Pi)^{-1} $$
Here we could have chosen to multiply by $\si (\Pi)^{-1}$ on the left. This would not change the properties of $\si_{\op{cov}}$. 

\begin{proof} The fact that the restriction $r(\cdot, k)$ to the diagonal of the renormalized Schwartz kernel $(  \bigl( \frac{2 \pi } {k} \bigr)^n K_k)$  belongs to $\Symb (M, \op{End} A) $ is actually a property of any operator $P$ in $\Fa_0$. Furthermore, if $P \in \Fa_{2m}$ and $\si_{2m} ( P ) ( \hb, 0 ) = \hb^m f_m$, then $r ( \cdot, k ) = k^{-m} f_m + \bigo ( k^{-m -1})$. In the case where $P$ is a Toeplitz operator, this proves by Lemma \ref{lem:usefull} that $P \in \Fa_{\infty} $ if and only if $r( \cdot, k) = \bigo( k^{-\infty})$. 

To prove that $\si$ is onto, observe that $\si ( \Pi_k M_{ k^{-m} g_m } \Pi_k) = k^{-m } g_m + \bigo ( k^{-m - 1}) $. Then we construct by successive approximation $(g( \cdot, k )) \in \Symb (M, \op {End} A)$ such that $\si (\Pi_k M_{g ( \cdot, k )} \Pi_k)$ is the series we want. 
\end{proof}

So we have defined two symbol maps, the contravariant and the covariant ones. These are total symbols in the sense that an operator with vanishing symbol is an operator in $\Fa_{\infty}$. Furthermore these symbol maps agree to first order. Indeed, for any Toeplitz operator, 
$$ \si_{\op{cov}} (T) = \si_0 ( T) + \bigo( \hb) , \qquad \si_{\op{cont}} ( T) = \si_0 (T)$$ 
The leading coefficient $\si_0 (T)$ is the {\em principal symbol } of the Toeplitz operator. It follows from Theorem \ref{theo:main1new2}, that if $T$ and $S$ are two Toeplitz operators, then 
$$ \si_0 (T S) = \si_0 ( T) . \si_0 ( S).$$
More generally, by Propositions \ref{prop:contravariant} and \ref{prop:covariant}, for any non negative integer $m$, $\si_{\op{cont}} (T) \in \bigo ( \hb^m)$ if and only if $\si_{\op{cov}}(T) \in \bigo ( \hb^m)$. If it is the case, $\si_{\op{cont}} ( T) = \si_{\op{cov}} ( T) + \bigo ( \hb ^{m+1}) $.   Furthermore, if $\si_{\op{cont} } (T) \in \bigo ( \hb^{m})$ and $\si_{\op{cont}} ( S) \in \bigo ( \hb^{ \ell}) $, then $\si_{\op{cont}} (T S) = \si_{\op{cont} } (T) . \si_{\op{cont} } (S) + \bigo ( \hb^{ m + \ell + 1} ) $. 

\subsection{Norms of Toeplitz operator} 

First we can characterize the Toeplitz operators in the residual ideal $\Fa_{\infty} \cap \To$ by their operator norm as follows. 
\begin{lemme} \label{lem:residual_Toeplitz}
Consider any operator family $T = (T_k : \Ci ( M, L^k \otimes A) \rightarrow \Ci ( M, L^k \otimes A) , \; k \in \N)$ satisfying $\Pi T \Pi = T$. Then  $T \in \Fa_{\infty} \cap \To$ if and only if $\|T_k \| = \bigo ( k^{-\infty})$. 
\end{lemme}

\begin{proof} 
  The implication follows from Proposition \ref{prop:norme_estimation}. To prove the converse, we apply Lemma \ref{lem:schw-kern} with $A = C = \Pi_k$ and $B = T_k$. We deduce the result from Formulas (\ref{eq:deriv_schwartz_kernel}) and (\ref{eq:estim_schwartz_kernel}).
\end{proof} 

We can also estimate the operator norm to first order in terms of the covariant or contravariant symbols. 
\begin{prop} \label{prop:norm_toeplitz}
Let $T$ be a Toeplitz operator. Then for any positive integer $m$,  
$$ \| T_k \| = \bigo ( k^{-m} ) \Leftrightarrow \si_{\op{cont}} ( T ) \in \bigo ( \hb^m)$$
If it is the case, one has
\begin{gather} \label{eq:norm_toep}
\| T_k \| = k^{ -m} \sup _{y \in M} | f ( y) | + \bigo ( k^{-m-1}) .
\end{gather}
where $f \in \Ci ( M, \op{End} A)$ is such that $ \si_{\op{cont}} ( T) = \hb ^m f +\bigo( \hb^{m+1})$. 
\end{prop} 

\begin{proof} 
By Lemma \ref{lem:residual_Toeplitz} and Proposition \ref{prop:contravariant}, we can modify $T$ by a $\bigo ( k^{-\infty})$ and assume $T = \Pi g( \cdot, k ) \Pi$ with $(g( \cdot , k ) ) \in \Symb ( M , \op{End} A)$. Then for any $k \in \N$, 
\begin{gather} \label{eq:norm_mult}
 \| T_k \| \leqslant \sup_{y \in M} | g ( y,k) | 
\end{gather}
Assume that $\si_{\op{cont} } ( T) =  \hb ^m f + \bigo ( \hb ^{m+1}) $. Then (\ref{eq:norm_mult}) implies that 
$$ \| T_k \| \leqslant k^ {-m} \sup_{y \in M} | f ( y) | +  \bigo ( k^{-m-1}).$$ 
Let us prove the converse inequality. Let $y \in M$ and $u \in A_y$ unitary and such that  $\sup_{y \in M} | f ( y) | = | \langle f(y).u, u \rangle |$. Introduce the vector $s_k \in \mathcal{H}_k$ given by $s_k(x)  = \Pi_k ( x, y) \cdot u$, where the dot stand for the contraction $\con{A}_y \times A_y \rightarrow \C$ induced by the metric of $A$. Then $\langle T_k s_k, s_k \rangle = \bar{u} \cdot T_k ( y,y)\cdot u$. Since $\si_{\op{cov}} (T) = \hb^m f + \bigo ( \hb^{m+1})$, we obtain that 
$$ \frac{ \langle T_k s_k, s_k \rangle}{ \langle s_k , s_k \rangle} = k^{ -m} \langle f( y ) u, u \rangle  + \bigo( k^{-m-1})  .$$
which proves (\ref{eq:norm_toep}). To show that conversely $ \| T_k \| = \bigo ( k^{-m} )$ implies that $\si_{\op{cont}} ( T ) \in \bigo ( \hb^m)$, we argue by induction on $m$ and use what we have just proved. 
\end{proof}

The adjoint of a formal series $ \sum \hb^ \ell f_\ell \in \Ci (M, \op{End} A) [[\hb]]$ is defined as $\sum \hb^\ell f_\ell^*$. 
The following proposition is easily proved.
\begin{prop} 
For any Toeplitz operator $T \in \To$, $T^*$ is a Toeplitz operator and
$$ \si ( T^* ) = \si (T) ^*, \quad \si_{\op{cont}}( T^*) = \si_{\op{cont}}( T)^*, \quad \si_{\op{cov}} (T^* ) = \si (\Pi )^{-1} \si_{\op{cov}} (T) \si ( \Pi)  . $$
If $A$ is a line bundle, then $ \si_{\op{cov}}( T^*) = \si_{\op{cov}}( T) ^*$.
\end{prop}

Propositions \ref{prop:contravariant}, \ref{prop:covariant} and \ref{prop:norm_toeplitz} imply Theorems \ref{theo:intro_contravariant} and \ref{theo:intro:principal} of the introduction, except for the formula giving the principal symbol of the commutator which is proved in Corollary \ref{cor:commutator_Poisson}.   Observe that the definition of Toeplitz operators given in the introduction is equivalent to the one given in this section by Proposition \ref{prop:contravariant} and Lemma \ref{lem:residual_Toeplitz}. 

\begin{rem} 
Consider the same data $(M,L,j ,A)$ as previously. Assume that $A$ is a subbundle of a Hermitian bundle $\mathbf{A}$.   
Let $\Pi \in \Fa_0 (M, L , j , \mathbf{A})$ be such that $\Pi^2 = \Pi$, $\Pi^* = \Pi$ and with symbol $1_A$. Here, we view $\op{End} A $ as a subbundle of $\op{End} \mathbf{A}$, by extending any endomorphism of $A_x$ to an endomorphism of $\mathbf{A}_x$ vanishing on the orthogonal of $A_x$. Such a projector exists. This follows from Theorem \ref{theo:existence-projector} by identifying $\Ci ( M , L^k \otimes A)$ with a subspace of $\Ci ( M , L^k \otimes \mathbf{A})$. 

Define $\To ( M , L, j  A, \Pi)$ as the space of operators $T \in \Fa_0 ( M, L, j , A)$ such that $\Pi T \Pi = T$. Then, considering as above $\op{End} A$ as a subbundle of $\op{End} \mathbf{A}$, Propositions \ref{prop:contravariant}, \ref{prop:covariant} and \ref{prop:norm_toeplitz} still hold. The proofs are the same.  \qed
\end{rem}

\subsection{Commutators} \label{sec:commutators}

Let $T \in \To$ be a Toeplitz operator whose principal symbol $f= \si_0 (T)$ takes scalar values. Then for any Toeplitz operator $S \in \To$, the principal symbol of $[T,S]$ vanishes, so that $ik [T,S]$ is a Toeplitz operator. Let us compute its principal symbol. We will first do the computation for a particular operator with principal symbol $f$.  

Introduce  a derivative $D_X : \Ci ( M , A) \rightarrow \Ci ( M, A)$ in the direction of the Hamiltonian vector field $X$ of $f$. Assume that $D_X$ preserves the metric of $A$.
Denote by $P_{X,k}$ the endomorphism:
\begin{gather} \label{eq:defpxk}
 P_{X,k} = ( \nabla^{L^k}_X \otimes \op{id} + \op{id} \otimes D_X) : \Ci ( M , L^k \otimes A) \rightarrow \Ci ( M , L^k \otimes A ) . 
\end{gather}
Then consider the operator $T'_k =  \Pi_k ( f + \frac{i}{k} P_{X,k} ) \Pi_k$. 
\begin{theo} \label{theo:com_system}
The family $ T' = ( T_k') $ is a Toeplitz operator with principal symbol $f$. For any Toeplitz operator $S$, $\bigl( ik [ T' , S] \bigr)$ is a Toeplitz operator with principal symbol $- D_X  \si_0 (S).$
\end{theo} 

\begin{proof}
The Schwartz kernel of $ \frac{i}{k} P_{X,k} \Pi_k $ is $( \frac{i}{k} P_{X,k} \boxtimes \op{id} )\Pi_k$. By Lemma \ref{lem:derEstim},  $( \frac{i}{k} P_{X,k} \Pi_k)$ belong to $\Fa_1 $. By Theorem \ref{theo:main1new}, $( \Pi_k \frac{i}{k} P_{X,k} \Pi_k)$ belong to $\Fa_1 $. By Lemma \ref{lem:usefull}, $( \Pi_k \frac{i}{k} P_{X,k} \Pi_k)$ belong to $\Fa_2 \cap \To$. Since $T = \Pi f \Pi$ modulo $\Fa_2 \cap \To$, $T$ is a Toeplitz operator with principal symbol $f$. 

For the second part, using that $\Pi S = S \Pi = S$, we have
$$ [ T_k' , S_k ] = \Pi_k [ f + \tfrac{i}{k} P_{X,k} , S_k ] \Pi_k .$$
By Lemma \ref{lem:lemquifaitmal},  $([ f + \tfrac{i}{k} P_{X,k} , S_k ])$ belongs to $\Fa_2$ and we know its principal symbol. So by Theorem \ref{theo:main1new}, $( [T_k', S_k ] )$ belongs to $\Fa_2$ and we can compute its principal symbol with Theorem \ref{theo:main1new2}. 
\end{proof} 

We can now answer our initial question. Let $f_1 \in \Ci ( M , \op{End} A)$ be such that $T \equiv T' + k^{-1} \Pi f_1 \Pi \mod \Fa_4$. Then the principal symbol of $ik [T, S] $ is $- D_X  \si_0 (S) + [ f_1, \si_0 (S)]$. 

\begin{cor} \label{cor:commutator_Poisson}
For any Toeplitz operators $T$ and $S$ with scalar valued principal symbols $f$ and $g$, the principal symbol of $ik [T,S]$ is the Poisson bracket of $f$ and $g$. 
\end{cor} 

\begin{proof} 
First, notice that since both $f$ and $g$ are scalar valued symbols, the principal symbol of $ik [T,S]$ only depends on $f$ and $g$. So we can assume that $T = \Pi_k ( f + \frac{i}{k} P_{X,k} ) \Pi_k$ and apply Theorem \ref{theo:com_system}. We deduce that the principal symbol of $ik [T, S]$ is $ - D_X .g =  \{ f, g \}$. At that point it can be interesting to note that modifying $D_X$ by a zero order term, $D_X.g$ does not change because $g$ is scalar. 
\end{proof} 

\begin{rem} 
In Theorem \ref{theo:com_system}, the operators $T'_k = \Pi_k ( f + \frac{i}{k} P_{X,k}) \Pi_k$ appears as a kind of normalized (up to $\Fa_4$) Toeplitz operator  with principal symbol $f$. So it is natural to ask what its covariant and contravariant symbols are up to $\bigo ( \hb^2)$. We know the answer only in the following case: assume that $j$ is integrable, $A$ is holomorphic, and let $\mathcal{H}_k $ consist of the holomorphic sections of $L^k \otimes A$ as in Remark \ref{rem:kahler}. Let $D_X = \nabla_X^A$ where $\nabla_X$ is the Chern connection of $A$. Then the covariant symbol of $T'$ is $ f$ up to $\bigo ( \hb^2)$  and the contravariant symbol is $f - \tfrac{\hb}{2} \Delta f $  where $\Delta$ is the Riemannian Laplacian associated to the metric $ \om ( \cdot , j \cdot)$, cf. \cite{oim_op}. \qed
\end{rem}


\subsection{Locality of symbols products} \label{sec:local-symb-prod}

In this section, to simplify the exposition, we assume that the auxiliary bundle $A$ has rank one so that $\op{End} A$ is the trivial bundle.  
We have defined three symbol maps 
$ \si, \si_{\op{cov}}, \si_{\op{cont}} : \To \rightarrow \Ci ( M )[[\hb]].$ 
These maps are onto and their kernel $\Fa_{\infty}  \cap \To$ is an ideal of $\To$. So the space $\Ci ( M ) [[\hb ]]$ inherits three associative products. Our goal is to prove that these product have the following form 
\begin{gather} \label{eq:loc_product} 
 \Bigr( \sum_{\ell \in \N} \hb^\ell f_\ell \Bigr) \circ  \Bigl( \sum_{m \in \N }  \hb^m g_m \Bigr) = \sum_{r\in \N}  \hb^r \sum_{m+ \ell + p = r } B_p ( f_\ell, g_m )  
\end{gather}
where for any $p$, $B_p$ is a bidifferential application from $\Ci( M ) \times \Ci( M )$ to $  \Ci( M )$ of order $2 p$, and $B_0 (f,g) = f.g$. 

Similarly, 
we will prove that the map $\Phi$ from $\Ci ( M ) [[\hb]]$ to itself  sending the contravariant symbol into the covariant symbol has the form 
\begin{gather} \label{eq:loc_equivalence}
 \Phi \bigl( \sum_{\ell \in \N}  \hb^\ell f_\ell \Bigr) = \sum_{ r \in \N} \hb^r \sum _{m + \ell = r } \Phi_m ( f_\ell ) 
\end{gather}
where for any $m$, $\Phi_m$ is a  differential operator acting on $\Ci ( M)$ of order $2 m$ and $\Phi_0 (f) = f$.  

We say that a product of $\Ci (M )[[\hb]]$ is {\em local} if it is of the form (\ref{eq:loc_product}). A bijection of $\Ci (M )[[\hb]]$ of the form (\ref{eq:loc_equivalence}) is called a {\em local equivalence}.

Introduce a coordinate system $(U, x_i)$ and a compact set $K$ contained in $U$. Denote by $\Ci_K(M)$ the space of functions supported in $K$. 

\begin{lemme} \label{lem:thelemme}
For any $P \in \Fa_m$, there exists a family $(P_{\al}, \al \in \N^n)$ of $\Fa$ such that for any $f \in \Ci _K (M)$ and $N\geqslant 0$, one has 
\begin{gather} \label{eq:8}
  P f = \sum_{|\al|\leqslant N} (\partial^\al f )    P_\al + R_N(f)  
\end{gather}
where $R_N(f) \in \Fa_{m+N+1}$. Furthermore,  for any $\al$, $P_{\al} \in \Fa_{m+|\al|}$.
\end{lemme}

\begin{proof} 
Introduce a function $\rho \in \Ci_0 ( U)$ which is equal to $1$ on a neighborhood of $K$. Define $P_{\al}$ as the operator with Schwartz kernel $ (y-x)^{\al} \rho(x) \rho(y) P(x,y) /\al!$, where $(x,y) \rightarrow P(x,y)$ is the Schwartz kernel of $P$.

Writing the Taylor expansion of $f(y)$ at $y = x$, one gets
\begin{gather} \label{eq:taylor}
 f(y) = \sum_{|\al|\leqslant N} (\partial^\al f ) (x) \frac{(y-x)^{\al}}{\al !} \rho(x) \rho(y) + r_N(x,y )
\end{gather}
where $r_N \in \Ci (M^2)$ vanishes to order $N+1$ along the diagonal. 
Multiplying (\ref{eq:taylor}) by $P(x,y)$, we obtain (\ref{eq:8}) with $R_N(f)$  the operator with Schwartz kernel $ r_N(x,y) P (x,y)$. Since $r_N$ vanishes to order $N+1$ along the diagonal $R_N(f) \in  \Fa_{m+N+1}$.
\end{proof}

\begin{theo} \label{theo:star_product}
The products corresponding to $\si$, $\si_{\op{cont}}$ and $\si_{\op{cov}}$ are local. The map $\Phi$ sending the covariant symbol to the contravariant one is a local equivalence. 
\end{theo}

\begin{proof} Using that $ \si ( k^{-1}T) = \hbar \si (T)$, one proves that the various products have the form (\ref{eq:loc_product}) without knowing that the $B_\ell$ are bidifferential. For the same reason, the map $\Phi$ has the form (\ref{eq:loc_equivalence}) and it remains to prove that the $\Phi_{\ell}$ are differential. 

Applying Lemma \ref{lem:thelemme} to $\Pi$, one gets a family $(\Pi_{\al})$ of $\Fa$. For any $f \in \Ci_K(M)$, one has $\Pi f \Pi = \sum ( \partial^\al f) \Pi_{\al} \Pi $. Since $\Pi_{\al} \Pi \in \Fa_{|\al|}$, the restriction to the diagonal of the Schwartz kernel of $\Pi_{\al} \Pi$ has the asymptotic expansion 
$$  (\Pi_\al \Pi )(x,x) = \Bigl( \frac{k}{2 \pi} \Bigr)^n  \sum_{\ell \geqslant |\al|/2 }k^{-\ell}  p_{\al, \ell} (x)  +\bigo(k^{-\infty}),$$
where the $p_{\al, \ell}$ are in $\Ci (M)$. 
As a consequence,
$$ \si ( \Pi f \Pi) =  \sum_{\al, \ell;\;  |\al| \leqslant 2 \ell} \hbar^{\ell} p_{\al, \ell}   \partial^\al f
$$  
From this local result, one deduces by using a partition of unity  that the map sending the contravariant symbol to the $\si$-symbol is a local equivalence. Since one obtains the covariant symbol from the $\si$-symbol by dividing by $\si ( \Pi)$, one deduces that $\Phi$ is a local equivalence. 

To show that the contravariant product is local, we argue similarly. One writes for any $f,g \in \Ci_K(M)$ that 
$$\Pi f \Pi g \Pi = \sum  (\partial^\al f) \Pi_\al\Pi g \Pi = \sum (\partial^\al f ) ( \partial^\be g) (\Pi_{\al} \Pi)_{\be } \Pi.$$  
Since $ (\Pi_{\al} \Pi)_{\be } \Pi$ belongs to $\Fa_{|\al| + |\be|}$, one has 
$$ ((\Pi_\al \Pi )_{\be} \Pi) (x,x) =\Bigl( \frac{k}{2 \pi} \Bigr)^n  \sum_{\ell \geqslant  |\al|/2 + |\be|/2}k^{-\ell}  q_{\al,\be, \ell}(x) +\bigo(k^{-\infty}).$$
Consequently, 
$$  \si ( \Pi f \Pi g \Pi ) = \sum_{  |\al| + | \be | \leqslant 2 \ell } \hbar^{\ell} q_{\al, \be, \ell}  (\partial^\al f) (\partial^\be g) $$
Furthermore, if $f,g \in \Ci (M)$ have disjoint supports, then the Schwartz kernel of $f \Pi g$ is uniformly a $\bigo ( k^{-\infty})$, so the same holds for $\Pi f \Pi g \Pi $ and consequently $\si ( \Pi f \Pi g \Pi ) = 0$. Using a partition of unity, one deduces that the contravariant product is local. Since one obtains the covariant and $\si$-symbol from the contravariant symbol by a local equivalence, the products of $\si$-symbols and covariant symbols are local too. 
\end{proof}

\subsection{Further estimates} \label{sec:further-estimates} 
We still assume that $A$ is a line bundle.  Denote by $(B_\ell ; \; \ell \in \N)$ the bidifferential operators of $\Ci ( M )$ corresponding to the product of contravariant symbols. For any $f$, $g \in \Ci(M)$, introduce 
$$R _N(f,g)  = \Pi f \Pi g \Pi - \sum_{\ell \leqslant N } k^{-\ell}  \Pi B_{\ell} ( f,g )\Pi.$$
By construction, $R_N(f,g) \in \Fa_{2(N+1)}$ and its norm is a $\bigo ( k^{-(N+1)})$. We will make explicit the dependence of this $\bigo$ in terms of $f$ and $g$. 

\begin{theo} \label{theo:sharp_estim}
For any $N$, there exists $C_N$ such that for any $f,g \in \Ci(M )$, one has $\| R_N( f,g) \| \leqslant C_N k^{-(N+1)} |f,g|_{2(N+1)}$.
\end{theo}

We say that an element $P$ of $\Fa$ is supported in a compact set $K$ of $M$ if for any $x \in M \setminus K$, there exists a neighborhood $V$ of $(x,x)$ such that the restriction of the Schwartz kernel of $P$ to $V$ is in $\bigo ( k^{-\infty})$ uniformly.

\begin{lemme} The space of operators in $\Fa$ supported in $K$ is a bilateral ideal of $\Fa$. 
\end{lemme}

\begin{proof} 
This follows from the fact that any operator in $\Fa$ has a Schwartz kernel in $\bigo ( k^n)$ whose restriction to any compact subset of $M^2$ not intersecting the diagonal is in $\bigo ( k^{-\infty})$.
\end{proof}

Consider a coordinate system $(U,x_i)$,  a compact set $K \subset U$ and $\rho \in \Ci_0 ( U)$ equal to $1$ on a neighborhood of $K$. 
For any $\al \in \N^n$, define the function $g_{\al} \in \Ci (M^2)$ which vanishes outside $U^2$ and is given on $U^2$ by 
$$ g_{\al}(x,y) = \rho(x) \rho(y) (x-y)^\al .$$  
Introduce a section $E$ satisfying the assumption of Lemma \ref{lem:sectionE} and let $Q_{\al}$ in $\Fa_{|\al|}$ be the operator with Schwartz kernel $(k/2\pi)^n E^k g_\al$. 

\begin{lemme} \label{lem:rep}
For any $N$,  any $P \in \Fa$ supported in $K$ has a unique representative $P_N$ modulo $\Fa_{N+1}$ of the form $$P_{N} = \sum_{\al, \ell} k^{-\ell} f_{\al, \ell} Q_\al$$  
where $(f_{\al, \ell}; \; |\al| +2 \ell \leqslant N, |\al| + 3 \ell \geqslant 0)$ is a family of $\Ci_K(M)$.
\end{lemme}
\begin{proof} 
The proof is an induction on $N$ by using that any operator of $\Fa$ supported in $K$ has a symbol supported in $K$ and that any function of $M^2$ vanishing to order $N$ along the diagonal can be written on $U^2$ on the form $\sum _{|\al| = N} f_{\al}(x) g_{\al} (x,y)$ up to some function vanishing to order $N+1$.  
\end{proof}

\begin{lemme} \label{lem:thelemmeplus}
The remainder in Lemma \ref{lem:thelemme} satisfies
$$ \| R_N(f) \| \leqslant C_N  |f|_{N+1} k^{-(m+N+1)/2} $$
for some constant $C_N$ independent of $k$ and $f$. 
\end{lemme}

\begin{proof} 
Observe that the function $r_N$ defined in (\ref{eq:taylor}) is supported in $M \times K$ and satisfies
\begin{equation} \label{eq:estimate_taylor}
\begin{split} 
 | r_N (x,y) | \leqslant  C_N |f|_{N+1} |x-y|^{N+1} \qquad \text{ on }U^2, \\
 |r_N ( x,y )| \leqslant  |f|_0 \qquad \text{ on } M^2 \setminus K^2 .
\end{split}
\end{equation}
Introduce a function $\rho' \in \Ci_0 (U)$ such that $\rho' = 1$ on a neighborhood of the support of $\rho$.
One has 
$$ R_N( f) = R'_N(f) + R''_N(f)$$ where  $R_N'(f)$ and $R_N''(f)$ are the operators obtained by multiplying the Schwartz kernel of $P$ by $ r_N(x,y) \rho'(x) \rho (y)$ and $r_N(x,y) (1 - \rho'(x)) \rho (y) $ respectively. 
One deduces by following the same method as in the proof of Proposition \ref{prop:norme_estimation}
$$ \| R_N( f) \| \leqslant C'_N  |f|_{N+1} k^{-(m+N+1)/2}, \qquad \| R_N''( f) \| \leqslant C_{N, \ell} |f|_0 k^{-\ell} $$
The first estimate follows from the fact that $\rho'(x) \rho (y)$ is supported on $U^2$ and by using the first estimate of   (\ref{eq:estimate_taylor}) instead of Equation (\ref{eq:maj_symb}), the second estimate follows from the fact that $(1 - \rho'(x)) \rho (y)$ vanishes identically on a neighborhood of the diagonal of $M^2$ and by using the second estimates of (\ref{eq:estimate_taylor}) instead of Equation (\ref{eq:maj_symb}).
\end{proof}
In the next two lemmas, we use the notations $P_\al$ and $P_N$ introduced in Lemmas \ref{lem:thelemme} and \ref{lem:rep} respectively. 
\begin{lemme} \label{lem:step1}
For any $f, g \in \Ci_K (M)$, the operators 
$$ S_N(f) = \Pi f \Pi -(\Pi f \Pi )_{2N+1} , \qquad S_N'(f,g) = \Pi f \Pi g \Pi - (\Pi f \Pi g \Pi)_{2N+1}  $$
satisfy $S_N( f) = |f|_{2N+2} \bigo ( k^{-(N+1)})$ and $S_{N}'(f,g) = |f,g |_{2N+2} \bigo ( k^{-(N+1)})$.
\end{lemme}

\begin{proof} 
By Lemmas \ref{lem:thelemme} and \ref{lem:thelemmeplus}, one has 
\begin{gather}\label{eq:4}
\Pi f \Pi =  \sum_{|\al| \leqslant 2N+1} ( \partial^\al f) \Pi_{\al} \Pi + |f|_{2(N+1)} \bigo ( k^{-(N+1)}) 
\end{gather}
Since $\Pi_{\al} \Pi = ( \Pi_{\al} \Pi )_{2N+1} + \bigo ( k^{-N-1})$, one obtains
\begin{gather} \label{eq:3}
 \Pi f \Pi =  \sum_{|\al| \leqslant 2N+1} ( \partial^\al f) (\Pi_{\al} \Pi)_{2N+1} + |f|_{2(N+1)} \bigo ( k^{-(N+1)}) 
\end{gather}
Since $\partial^\al f$ is supported in $K$, $( \partial^\al f) (\Pi_{\al} \Pi)_{2N+1} = \bigl( (\partial ^\al f) \Pi_{\al} \Pi \bigr)_{2N+1} $. 
Furthermore observe that the remainder in (\ref{eq:3}) is in $\Fa _{2(N+1)}$. This proves that 
$$ \Pi f \Pi = ( \Pi f \Pi)_{2N+1}  + |f|_{2(N+1)} \bigo ( k^{-(N+1)}) $$
The proof of the second estimate is similar. Using Lemmas \ref{lem:thelemme} and \ref{lem:thelemmeplus}, one has
\begin{gather} \label{eq:5} 
 \Pi_{\al}\Pi g \Pi = \sum_{|\be| \leqslant 2N+1} ( \partial^\be g) (\Pi_{\al} \Pi)_{\be} \Pi  + |g|_{2(N+1)} \bigo ( k^{-(N+1)})  
\end{gather}
Using Equations (\ref{eq:4}) and (\ref{eq:5}), one gets
\begin{gather} \label{eq:6}
\Pi f \Pi g \Pi = \sum_{|\al|, |\be| \leqslant 2N+1} (\partial^{\al}f) ( \partial^\be g) (\Pi_{\al} \Pi)_{\be} \Pi  + |f,g|_{2(N+1)} \bigo ( k^{-(N+1)})  
\end{gather}
The remainder in (\ref{eq:6}) being in $\Fa_{2(N+1)}$, one gets
$$ \Pi f \Pi g \Pi =  \bigl( \Pi f \Pi g \Pi \bigr)_{2N+1}  + |f,g|_{2(N+1)} \bigo ( k^{-(N+1)})  $$
which was to be proved.
\end{proof}

\begin{lemme} \label{lem:step2}
One has for any $f,g \in \Ci_K( M )$
$$ \Pi C_N ( k,f,g) \Pi = \Pi f \Pi g \Pi +  |f,g|_{2(N+1)} \bigo ( k^{-(N +1)}) $$
 where $C_N (k,f,g) = \sum_{\ell \leqslant N} k^{-\ell} B_{\ell} ( f,g)$. 
\end{lemme}
\begin{proof} 
Since each $B_\ell$ is of order $2 \ell$, one deduce from Lemma \ref{lem:step1} that 
$$ \Pi B_{\ell} ( f,g) \Pi = ( \Pi B_{\ell}(f,g)  \Pi ) _{2(N-\ell ) +1} +  |f,g|_{2(N+1)} \bigo ( k^{-(N -\ell +1)}) $$
Multiplying by $k^{-\ell}$ and summing over $\ell$, one obtains
\begin{gather} \label{eq:9}
 \Pi C_N ( k,f,g) \Pi =  \bigl( \Pi C_N ( k,f,g) \Pi \bigr)_{2N+1} + |f,g|_{2(N+1)} \bigo ( k^{-(N+1)}) 
\end{gather}
Since $ \Pi C_N ( k,f,g) \Pi \equiv \Pi f \Pi g \Pi $ modulo $\Fa_{2(N+1)}$, one has $$\bigl( \Pi C_N ( k,f,g) \Pi \bigr)_{2N+1}  = \bigl( \Pi f \Pi g \Pi \bigr)_{2N+1}$$
The conclusion follows from Equation (\ref{eq:9}) and the second estimate of Lemma \ref{lem:step1}.
\end{proof}

\begin{lemme} \label{lem:out_diag}
For any disjoint compact subsets $K$, $K'$ of $M$, for any $N$, there exists $C_N$ such that for any $f \in \Ci _K( M )$, $g \in \Ci_{K'} ( M)$ one has 
$$ \| \Pi f \Pi g \Pi \| \leqslant C_N |f,g|_0 k^{-N}$$
\end{lemme}
\begin{proof} 
It is a consequence of the fact the Schwartz kernel of $\Pi$ is uniformly a $\bigo ( k^{-\infty})$ on $K \times K'$.
\end{proof}

\begin{proof}[Proof of Theorem  \ref{theo:sharp_estim}]
Let $(\varphi_{\ga} , \ga \in \Ga)$ be a partition of unity of $M$ with $\Ga$ a finite set and such that the support of each $\varphi_{\ga}$ is contained in the domain $U_\ga$ of a coordinate system. Choose functions $\Psi_{\ga} \in \Ci (M)$ supported in $U_{\ga}$ and equal to $1$ on a neighborhood of $\supp \varphi_{\ga}$. 
Let $S''_N(f,g) = \Pi f \Pi g \Pi - \Pi C_N(k,f,g) \Pi$. Then we have 
$$ S''_N(f,g) = \sum_\ga S''_N( f \varphi_\ga, g) = \sum_\ga S''_N( f \varphi_\ga, g \Psi_{\ga} )+ \sum_\ga S''_N( f \varphi_\ga, g (1-\Psi_{\ga}) ) $$
By applying Lemma \ref{lem:step2} to the compact set $K = \supp \Psi_{\ga}$ we obtain  
$$ S''_N( f \varphi_\ga, g \Psi_{\ga} ) =  |f,g|_{2(N+1)} \bigo ( k^{-(N +1)}) $$
By applying Lemma \ref{lem:out_diag} to the compact sets $\supp \varphi_{\ga}$ and $\supp ( 1 - \Psi_{\ga})$, we obtain
$$  S''_N( f \varphi_\ga, g (1-\Psi_{\ga}) ) =  |f,g|_{2(N+1)} \bigo ( k^{-(N +1)}) $$
which concludes the proof. 
\end{proof}

\appendix
\section{Spin-c Dirac quantization} 

Consider a compact symplectic manifold $(M, \om)$ with a compatible almost-complex structure $j$ and a prequantum bundle $L \rightarrow M$. Let $A \rightarrow M$ be a Hermitian vector bundle endowed with a connection $\nabla^A$. 

$M$ has a natural Riemannian metric $g = \om ( \cdot, j \cdot)$. Consider the corresponding Clifford bundle $\op{C} ( T^*M)$ and the Spinor bundle 
$$ S = \wedge \bigl ( (T^*M)^{1,0} \bigr) $$
where the action of $(T^*M)^{1,0}$ is by exterior multiplication and that of $(T^*M)^{0,1} $ is by contraction. Choose any Hermitian connection on the line bundle $\op{det}\bigl(( TM)^{1,0} \bigr)$ and let $\nabla^{S} $ be the corresponding Clifford connection of $S$.  

For any $k \in \N$, let $\mathbf{A}_k = L^k \otimes A \otimes S$. Let $\mathbf{\nabla}^k$ be the connection of $\mathbf{A}_k$ induced by the ones of $L$, $A$ and $S$. The corresponding spin-c Dirac operator 
$$   D_k : \Ci ( M , \mathbf{A}_k ) \rightarrow \Ci ( M , \mathbf{A}_k) $$
is given by the covariant derivative $\mathbf{\nabla}^k : \Ci ( M, \mathbf{A}_k) \rightarrow \Ci ( M , \mathbf{A}_k \otimes T^*M)$ composed with the Clifford action. For more details on this construction, we refer the reader to \cite{MaMa} Chapter 1.3 and \cite{Du}, Chapters 3 to 5. 

Denote by $\Pi_k$ the orthogonal projector of $\Ci ( M , \mathbf{A}_k)$ onto the kernel of $D_k$. 

\begin{theo} \label{theo:spin-c-dirac}
The family $( \Pi_k, \; k \in \N^*)$ belongs to $\Fa _0 (M, L, j, A \otimes S)$. Its symbol is the section $p_A$ of $\op{End} (A \otimes S)$, such that for any $x \in M$, $p_A ( x)$ is the orthogonal projector onto $A_x \otimes \C \subset A_x \otimes S_x$.  
\end{theo} 
 
We will deduce this theorem from the results in the Section 8. of \cite{MaMa}. First by Theorem 8.1.2 in  \cite{MaMa}, the Schwartz kernel of $ \Pi_k$ is a $\bigo ( k^{-\infty})$ outside the diagonal. Furhtermore, by Theorems 8.1.1.2 and 8.1.4 of \cite{MaMa}, the  Schwartz kernel of $ \Pi_k$ is in $\bigoinf ( k^{n})$. 
Let $x \in M$. Denote by $B_x(\ep)$ the open ball of $T_{x} M$ centered at $0$ with radius $\ep$. If $\ep$ is sufficiently small, the exponential map of the metric $g$ induces a diffeomorphism $\varphi_{x} : B_x( \ep) \rightarrow M$ onto an open set of $M$. Consider the isomorphism $$\varphi_{x}^* \mathbf{A}_k \simeq B_x(\ep) \times \mathbf{A}_{k,x} $$ obtained by doing parallel transport along the radii $ [ 0,1]  \ni t \rightarrow t \xi$, $\xi \in B_x(r)$. Then
\begin{xalignat*}{2}
 ( \varphi_{x} \times \varphi_{x} )^* ( \mathbf{A}_k \boxtimes \con{\mathbf{A}}_k ) & \simeq  (B_x (\ep) )^2 \times (\mathbf{A}_{k,x} \otimes \con{\mathbf{A}}_{k,x} ) \\
 & \simeq  (B_x (\ep) )^2 \times (\mathbf{A}_{x} \otimes \con{\mathbf{A}}_{x} ) 
\end{xalignat*}
where $\mathbf{A}_x = A_x \otimes S_x$. Here we have used the identification $L^k_x \otimes \con{L}^k_{x} \simeq \C$ induced by the metric.
By Theorem  8.1.4 of \cite{MaMa},  we have for any $\ell \in \N$, $\xi, \eta \in B_x (\ep)$, 
\begin{gather} \label{eq:MaMa}
 \bigl( ( \varphi_{x} \times \varphi_{x} )^*\Pi_k \bigr) ( \xi, \eta ) = e ^{ - k/4 ( |\xi|^2 + | \eta|^2 - 2 \xi \cdot \eta ) } \sum_{r = 0 }^{\ell} P_{r} ( x, k^{1/2} \xi, k^{1/2} \eta) + r_{\ell} (x, \xi, \eta)
\end{gather}
Here for any $r$, the map $P_r (x, \cdot) : (T_x M \oplus T_x M) \rightarrow \mathbf{A}_{x} \otimes \con{\mathbf{A}}_{x}$ is polynomial with degree $3r$ and it has the same parity as $r$. Furthermore the remainder $r_{\ell} ( x, \cdot)$ satisfies for some constant $C_\ell>0$ and $N_\ell$
\begin{gather} \label{eq:est_reste}
 r_{\ell} ( x, \xi, \eta ) = \bigo \Bigl( k^{-( \ell +1 ) /2 }  ( 1 + k^{1/2} ( |\xi | + | \eta | )^{N_\ell}) e^{- k^{1/2} | \xi - \eta| /C_\ell} \Bigr) + \bigo( k^{-\infty}) 
\end{gather}
where the $\bigo$ are uniform with respect to $\xi$, $\eta$ and $k$. By  Theorem  8.1.4 of \cite{MaMa}, we also know that the $P_r$'s depend smoothly in $x$ and that the previous estimates are uniform with respect to $x$. 

The map $\varphi : TM \rightarrow M^2$ sending $(x,\xi)$ into $(\varphi_x ( \xi) , x)$ is a diffeomorphism from a neighborhood of the null section to a neighborhood of the diagonal. Consider the section $E$ of $\varphi^* (L \boxtimes \con {L} )$ given by
$$ E( x, \xi) = e^{ -|\xi|^2 /4} s(x,\xi) \otimes \con{s( x)} $$
where $s(x)$ is any unitary vector of $L_x$ and $s(x, \xi) \in L_{\varphi_x( \xi)}$ is obtained by parallel transporting $s(x)$ along the curve $ [0,1] \ni t \rightarrow (x, t \xi)$.  Similarly, we may identify $(\varphi^* \mathbf {A})_{x,\xi} $ with $\mathbf{A}_x$ by doing parallel transport along the curve $[0,1] \ni t \rightarrow (x, t \xi)$. Setting $\eta =0$ in equation (\ref{eq:MaMa}), we obtain that 
$$(\varphi ^* \Pi ) (x, \xi ) = \Bigl( \frac{k}{2 \pi} \Bigr) ^n E^k ( x, \xi) \sum_{r = 0 }^{\ell} P_{r} ( x, k^{1/2} \xi, 0) + r_{\ell} (x, \xi, 0 )
$$
Since $t^{N_\ell} e^{ -t /C_{\ell}}$ is bounded over $\R^+$, we have that
$$  r_{\ell} ( x, \xi , 0) = \bigo ( k^{-(\ell+1)/2} ).
$$
By Proposition \ref{prop:altern}, the section $E$ satisfies the required conditions. We conclude the proof of Theorem \ref{theo:spin-c-dirac} by applying Proposition \ref{prop:MM}.


\bibliographystyle{alpha}
\bibliography{bibnew}

\end{document}